\documentclass[11pt,english,a4paper]{amsart}

\usepackage[english]{babel}
\usepackage[T1]{fontenc}

\usepackage{amsmath,amssymb,amsfonts,bm}
\usepackage{mathrsfs}

\usepackage{xcolor}
\usepackage{graphicx}

\usepackage{dsfont}
\usepackage{stmaryrd}
\usepackage[mathscr]{euscript}
\usepackage{cancel}
\usepackage{mathabx}

\usepackage{nicematrix}
\usepackage[all]{xy}
\entrymodifiers={+!!<0pt,\fontdimen22\textfont2>}

\usepackage{array}
\usepackage{float}
\usepackage{placeins}
\usepackage{verbatim}
\usepackage{enumerate}
\usepackage{caption}

\usepackage{pbsi}

\usepackage[
  pagebackref,            % renvoi biblio -> texte
  colorlinks=true,
  linkcolor=blue,
  citecolor=blue,
  urlcolor=red,
  hypertexnames=true
]{hyperref}

\newtheorem{thm}{Theorem}[section]
\newtheorem*{thm*}{Theorem}

\newtheorem{theorem}[thm]{Theorem}

\newtheorem{lem}[thm]{Lemma}
\newtheorem{lemma}[thm]{Lemma}

\newtheorem{prop}[thm]{Proposition}
\newtheorem{proposition}[thm]{Proposition}

\newtheorem{cor}[thm]{Corollary}

\newtheorem{defi}[thm]{Definition}

\newtheorem{problem}[thm]{Problem}

\newtheorem{ques}[thm]{Question}

\theoremstyle{remark}
\newtheorem{remark}[thm]{Remark}

\newcommand{\indic}{\mathbf{1}}

\newcommand{\T}{\mathbb{T}}
\newcommand{\D}{\mathbb{D}}

\newcommand{\id}{\mathrm{id}}

\newcommand{\A}{\mathcal{A}}

\newcommand{\I}{\mathcal{I}}

\newcommand{\N}{\mathbb{N}}

\renewcommand{\P}{\mathbb{P}}
\newcommand{\R}{\mathbb{R}}

\newcommand{\C}{\mathbb{C}}

\newcommand{\U}{\mathcal{U}}

\newcommand{\Z}{\mathbb{Z}}

\newcommand{\veps}{\varepsilon}

\newcommand{\ol}[1]{\overline{#1}}

\renewcommand{\Re}{\operatorname{Re}}

\newcommand{\la}{\langle}
\newcommand{\ra}{\rangle}

\newcommand{\limk}{\lim_{k\to\infty}}
\newcommand{\limn}{\lim_{n\to\infty}}

\newcommand{\Ttnull}{\mathcal{T}_{T_0}}
\newcommand{\Eins}{\mathbf{1}}

\newcommand{\Tlim}{\mathcal{L}_{T}}
\newcommand{\Tlimstr}{\mathcal{L}_{T}^{\text{str}}}
\newcommand{\Tnulllim}{\mathcal{L}_{T_0}}

\newcommand{\TT}{\mathcal{T}}
\newcommand{\const}{\text{const}}
\newcommand{\Pconst}{P_{\text{const}}}

\newcommand{\wot}{\texttt{WOT}}
\newcommand{\sott}{\texttt{SOT}\mbox{$_{*}$}}
\newcommand{\sot}{\texttt{SOT}}
\newcommand{\sote}{\texttt{SOT}\mbox{$^{*}$}}

\newcommand{\Cp}{\mathcal{C}_p} % positive contractions on L^p
\newcommand{\Ip}{\mathcal{I}_p} % positive isometries on L^p
\newcommand{\Up}{\mathcal{U}_p} % positive bijective isometries on L^p

\newcommand{\Contp}{\mathrm{Cont}_p}
\newcommand{\GG}{\mathcal{G}} % nonsingular transformations
\newcommand{\wTau}{\omega_\tau} %dérivée radon nikodym

\title{%Typical asymptotics %of operators and measure-preserving transformations}
Weak limit semigroup
in operator theory and ergodic theory}

\author{Tanja Eisner}
\address{Institute of Mathematics, University of Leipzig, P.O. Box 100 920, 04009 Leipzig, Germany}
\email{eisner@math.uni-leipzig.de}
\author{Valentin Gillet}
\address{Laboratoire Paul Painlev\'e, UMR 8524, Universit\'e de Lille, Cit\'e Scientifique, B\^atiment M2, 59655 Villeneuve d'Ascq Cedex (France) }
\email{valentin.gillet@univ-lille.fr}


\begin{document}

%%%%%%%%%%%%%%%%%%%%%%%%%
%
%                       Transformations
%
%%%%%%%%%%%%%%%%%%%%%%%%%

\maketitle

\begin{abstract}
We study the \emph{weak limit semigroup} of an operator $T$, i.e., the set of all operators being weak limit points of the powers of $T$, in three different but related contexts: Koopman operators of measure-preserving transformations, contractions/isometries/unitaries on separable Hilbert spaces and positive operators on $L^p$-spaces. 
Hereby we focus on finding large subsets of the weak limit semigroup, in particular in the generic case. %TODO more? 
\end{abstract}
%\vspace{0.5cm}

\vspace{0.4cm}

\centerline{\Small \emph{The only thing we know about the future is that it will be different.}}
\centerline{\Small \textsc{Peter Drucker}}

\section{Introduction}

In this paper we study the following 
question. 
\begin{ques}
Given an operator $T$, 
which operators can appear as limits of a subsequence of powers of $T$ with respect to the weak operator topology? I.e., determine $V$ such that there exists a subsequence\footnote{By a subsequence of $\N$ we always mean a strictly monotone sequence in $\N$.} $(n_k)$ of $\N$ with weak-$\displaystyle \limk T^{n_k}=V$.
\end{ques}
We call such operators $V$ \emph{weak limit operators} for the operator $T$, and the set of all such operators, denoted by $\Tlim$, the \emph{weak limit semigroup}.

We will study the weak limit semigroup of $T$ in three contexts: $T$ being the Koopman operator on $L^2[0,1]$ of a measure-preserving transformation, a contraction/isometry/unitary on a separable Hilbert space $H$ and a positive operator on $L^p[0,1]$, $p\in[1,\infty)$. Note that for Koopman operators and isometries/unitaries, weak limit operators are in general not Koopman or isometries/unitaries anymore - for example, the orthogonal projection onto the constants or zero, respectively, is often such a weak limit operator. %Usually considers the \emph{strong limit semigroup} of $T$, i.e., the set of strong limit 

As the recent result\footnote{Note that Solecki only considered the limit operators $V$ which are again Koopman.}  
of Solecki \cite{Solecki23} from ergodic theory indicates, typically the weak limit semigroup has a complex structure which is not yet understood and it is difficult to determine all limit operators. We focus on finding large subsets of the weak limit semigroup, in particular in the generic case.

In ergodic theory, there are various results for the intersection of the weak limit semigroup with the Koopman operators of invertible transformations (being precisely the set of limit operators for the strong$^*$ operator topology\footnote{A net of operators $(T_\alpha)$ is said to converge to $T$ \emph{strongly$^*$} if $\displaystyle \lim_{\alpha}T_\alpha=T$ and $\displaystyle \lim_{\alpha}T^*_\alpha=T^*$ strongly.}). For example, Chacon and  Schwartzbauer \cite{ChaconSchwartzbauer69} showed\footnote{see Katok and Stepin \cite[Thm. 1.1]{KatokStepin67} for the fact that the transformations considered by Chacon, Schwartzbauer are generic} that for a generic transformation $T$, all invertible Koopman operators commuting with $T$ belong to $\Tlim$. (Note that all weak limit operators of course commute with $T$.) In his famous Weak Closure Theorem, generalizing a result of del Junco \cite{delJunco78}, King \cite{King86} showed\footnote{see Ryzhikov \cite{Ryzhikov07} for an alternative proof} that all rank one transformations have this property (even without invertibility of commuting Koopman operators) and added that for rank one transformations, the centralizer (in Koopman operators) is either trivial\footnote{i.e., equals $\{T^n:\, n\in \Z\}$} or uncountable. An example of a trivial centralizer (in not necessarily invertible Koopman operators) is the Chacon transformation \cite{Chacon}, see del Junco \cite{delJunco78}. For an abstract representation of the intersection of the weak limit semigroup of a generic transformation with (invertible) Koopman operators, see Solecki \cite{Solecki14}.

On the other hand, Halmos \cite{Halmos44}, \cite[pp. 77--80]{Halmos-book}  showed that the orthogonal projection onto the constants 
%which is not even an isometry, 
is generically in the weak limit semigroup. 
%\footnote{more precisely, a generic transformation is weakly mixing but not mixing, see also Rokhlin \cite{Rokhlin48}.}. 
Further operators which are not isometries and are typically in the weak limit semigroup were 
%discovered
introduced by Oseledets \cite{Oseledets69}, studied 
by Stepin \cite{Stepin86} and Katok and Stepin \cite{KatokStepin67}, see also Katok \cite[Chapter I]{Katok03}, and revisited by  Ryzhikov \cite{Ryzhikov07b} (see Theorems \ref{thm:stepin} and \ref{thm:generic-transf} below). 
%(Note that the notion of $(1-\lambda)$-mixing was in fact introduced by Oseledets \cite{Oseledets69}.) 
For some transformations, the weak closure of the powers of it was completely determined. This is the case for the %famous 
Chacon transformation $T$ for which the weak closure of the powers %contains 
is the union of 
the orthogonal projection onto constants and some explicit family of 
polynomials of $T$, see de la Rue, Janvresse, Prikhod'ko and Ryzhikov \cite{delaRue-Janvresse-Prikhodko-Ryzhikov15}. For more examples see Ryzhikov \cite{Ryzhikov11,Ryzhikov12}. 

The problem of describing the structure of the weak limit semigroup of unitary operators dates back to West \cite{West68} where, building on the work of De Leeuw and Glicksberg \cite{Deleeuw-Glicksberg61}, he studied the structure of weakly compact monothetic semigroups of operators in terms of the spectral properties of the generating elements. Brown and Moran \cite{Brown-Moran71} obtained more information concerning the pathologie of these types of semigroups using techniques from harmonic analysis. One of West's motivations was to decide whether it is possible to construct compact monothetic semigroups with infinitely many idempotents. Brown and Moran showed that this is indeed the case, by proving that the weak closure of the positive powers of a unitary operator can have its idempotent subsemigroup arbitrarily prescribed and can be finite, countably or even uncountably infinite, but there is no control over its structure. Using Riesz products on the unit circle, Brown and Moran \cite{Brown-Moran74} also constructed unitary operators with the property that the weak closure of their positive powers decomposes the space in a prescribed way without losing control of the full semigroup.  O'Brien \cite{OBrien25} studied 
the W$^*$-algebra generated by the weak limit semigroup 
%in his study of 
of a unitary operator to characterise 
weak stability.

We will also consider the case of positive operators on $L^p$-spaces being intermediate between Koopman operators and  general contractions. Note that the set of (Koopman operators of) invertible measure-preserving transformations is weakly closed and nowhere dense in the space of positive contractions of $L^2[0,1]$, see, e.g., Theorem \ref{theoremevaleursproprestypicalposcontrWOT}. Also positive contractions of $L^2[0,1]$ are clearly weakly closed and nowhere dense in the space of all contractions of $L^2[0,1]$, so we have a chain of spaces which are closed and nowhere dense in the next one.  

We will see that %the structure of 
the weak limit semigroup of a typical positive contraction of $L^2[0,1]$ has similar properties to those of a typical transformation  (more precisely, of its restriction to $L^2_0[0,1]$) and a typical Hilbert space operator. 

Several authors investigated the approximation of positive and bi-stochastic operators on $L^p[0,1]$ such as Brown \cite{Brown66}, Grząślewicz \cite{Gralewicz90}, Iwanik \cite{Iwanik80} and Kim \cite{Kim68}. Our study will rely on these works. It will also rely on the works done on generic properties of non-singular transformations on $[0,1]$ by Chacon and Friedman \cite{ChaconFriedman65}, Choksi and Kakutani \cite{ChoksiKakutani79}, Choksi and Nadkarni \cite{ChoksiNadkarni90} and Ionescu \cite{Ionescu65}, on the works of Katok and Stepin on transformations admitting a periodic approximation \cite[Part. 1]{Katok03}, \cite{KatokStepin67} and \cite{Stepin86}, and on the work of Ryzhikov on the structure of the weak closure of powers of a generic transformation \cite{Ryzhikov07b}.
%
%Several authors investigated the approximation of positive and bi-stochastic operators of $L^p[0,1]$ such as Brown \cite{Brown66}, Grząślewicz \cite{Gralewicz90}, Iwanik \cite{Iwanik80} and Kim \cite{Kim68}. This motivates our study of the weak limit semigroup of a typical positive contraction of $L^p[0,1]$. Since positive operators are more general than Koopman operators, it is natural to expect different phenomena compared to the structure of the weak limit semigroup of a generic transformation.}

This paper is aimed for different communities of mathematicians and thus partially uses different terminology. For example, in the ergodic theory context we call a property \emph{generic} if the set of  transformations satisfying this property forms a comeager set in all (invertible) transformations, while in the operator theoretic context we call analogous properties \emph{typical}. Note that we try to distinguish dense $G_\delta$ sets from general comeager sets when possible. 

The study of generic (invertible) transformations has long and rich history in ergodic theory going back to Halmos \cite{Halmos44} and Rokhlin \cite{Rokhlin48} who showed that a generic transformation is weakly mixing but not mixing. Since then many exciting results have been found, see, e.g., Rokhlin \cite{Rokhlin59}, Katok \cite{Katok85}, Stepin \cite{Stepin86}, Choksi and Nadkarni \cite{ChoksiNadkarni90}, del Junco and Lemańczyk \cite{delJuncoLemanczyk92}, Nadkarni \cite{Nadkarni-book},  King \cite{King00}, Ageev \cite{Ageev03}, de la Rue and de Sam Lazaro \cite{delaRue-deSamLazaro03}, Stepin and Eremenko \cite{StepinEremenko04}, Ryzhikov \cite{Ryzhikov24} in addition to the above references, cf.~also \cite{E25}.  

Typical behavior of more general operators than Koopman is less well studied. For typical properties of general unitaries/isometries/contractions on separable Hilbert spaces see  Eisner and Ser\'eny \cite{EisnerSereny08}, Eisner \cite{E10,E-book}, Eisner and Mátrai \cite{EisnerMatrai13}, Zorin-Kranich \cite{Pavel}, Eisner and Radl \cite{EisnerRadl22}, 
 Grivaux, Matheron and Menet \cite{GrivauxMatheronMenet21,GrivauxMatheronMenet21b,GrivauxMatheronMenet22}, Grivaux and L\'opez-Mart\'inez \cite{GrivauxLopez-Martinez23}. Typical positive operators were studied in Bartoszek and Kuna \cite{BartoszekKuna05}, Gillet \cite{gillet1,gillet2}.

\subsection*{Organization of the paper}
After discussing notations and general properties of the weak limit semigroup in Section \ref{sec:prelim}, we show in Section \ref{sec:rig} that the only fixed (i.e., not depending on $T$) operators in the weak limit semigroup of a generic transformation $T$ are of the form $\lambda I+(1-\lambda)\Pconst$, $\Pconst$ being the orthogonal projection onto the constants, for $\lambda\in[0,1]$. These operators are known to be in the weak limit semigroup for a generic $T$ by a result of Stepin, see Theorem \ref{thm:stepin}. Analogously, we show that the only fixed operators in the weak limit semigroup of a typical unitary/isometry/contraction are of the form $\lambda I$, $\lambda\in\ol{\D}$, $\D$ being the open unit disc. Moreover, such operators are indeed in the weak limit semigroup of a typical unitary/isometry/contraction. For positive $L^p$-operators, we also show that the only fixed positive operators in the weak limit semigroup of a typical positive contraction/positive isometry/\allowbreak positive invertible isometry of $L^p[0,1]$ are of the form $\lambda I$, $\lambda \in [0,1]$. 

In Section \ref{sec:ex-op}, we present a class of examples of a unitary operator $T$ with maximal weak limit semigroup, i.e., being the set of all contractions which commute with $T$. The construction uses continuous measures on Kronecker sets. We then use this example in Section \ref{sec:ex-transf} to determine the weak limit semigroup for Gaussian processes with such spectral measures and see that among Gaussian processes, the weak limit semigroup of such processes is maximal.  

Building on the properties of the weak limit semigroup of a generic transformation that we outline in Section \ref{sec:gen-transf} (basically done by Stepin, Katok and Ryzhikov), we study the weak limit semigroup of a typical Hilbert space contraction/isometry/unitary $T$ in Section \ref{sec:genericasymptoticsforoperators}. We begin the section by refining the result from \cite{E10} by showing that unitaries form a dense $G_\delta$ set in isometries and contractions and isometries  form a dense $G_\delta$ set in contractions. Here, we endow unitary operators with the strong$^*$, isometries with the strong and contractions with the weak operator topology. We further show that invertible operators $T$ with the weak limit semigroup containing all operators of the form 
$$\displaystyle \sum_{n=-\infty}^\infty a_nT^n \text{ with } (a_n)_{n \in \Z} \subseteq \C \text{ satisfying }\displaystyle\sum_{n=-\infty}^\infty |a_n|\leq 1
$$
form a dense $G_\delta$ (resp., comeager) set in the space of unitary and isometric operators (resp., contractions).

% by showing that it contains many functions of $T$. 

Section \ref{Section-genericposoperators} is devoted to the study of the spectral properties of a typical positive contraction of $L^p[0,1]$ with applications to the study of the weak limit semigroup of such a typical positive contraction. 
Building on the work of Grząślewicz, we deduce that the set of positive invertible isometries is dense $G_\delta$ in the set of positive contractions of $L^p[0,1]$ for $1 < p < \infty$ with respect to the weak operator topology. Using the works on generic properties of non-singular transformations on $[0,1]$, the works of Katok and Stepin and the work of Ryzhikov, we prove that positive invertible operators $T$ with the weak limit semigroup containing all operators of the form 
$$\displaystyle \sum_{n=-\infty}^\infty a_nT^n \text{ with } (a_n)_{n \in \Z} \subseteq [0,1] \text{ satisfying }\displaystyle\sum_{n=-\infty}^\infty a_n\leq 1
$$
form a dense $G_\delta$ (resp., comeager) set in the space of all positive isometries and positive invertible isometries of $L^p[0,1]$ with respect to the strong operator topology (resp., positive contractions with respect to the weak operator topology). We also show that the weak limit semigroup of a typical positive contraction of $L^p[0,1]$ (with respect to the weak operator topology) does not contain any rank-one positive operator, contrary to the weak limit semigroup of a generic transformation. Building on the work of Iwanik, we prove that positive isometries of $L^1[0,1]$ with the weak limit semigroup containing all operators of the form $$\displaystyle \sum_{n = 0}^{N} a_n T^n  \text{ with } N \geq 0 \; \textrm{and} \; (a_n)_{0 \leq n \leq N} \subseteq [0,1] \text{ satisfying } \displaystyle \sum_{n = 0}^{N} a_n = 1 $$ is dense with respect to the strong operator topology.

We end the article with some open problems related to our study.

\medskip

\textbf{Acknowledgements.} We are deeply grateful to Sophie Grivaux for the key example from Proposition \ref{prop:ex-Sophie}. The first author thanks Sohail Farhangi for inspiring discussions on $\alpha$-mixing which led to the idea to write this paper. 

The second author acknowledges the support of the CDP C2EMPI, as well as the French State under the France-2030 programme, the University of Lille, the Initiative of Excellence of the University of Lille, the European Metropolis of Lille for their funding and support of the R-CDP-24-004-C2EMPI project.

\section{Notation and Preliminaries}\label{sec:prelim}

\subsection{Measure-preserving transformations}
 
We denote by $\mathcal{T}$ the set of all invertible measure-preserving transformations on $[0,1]$ with respect to the Lebesgue measure,\footnote{One can equivalently take any  standard non-atomic probability space instead of [0,1] endowed with the Lebesgue measure, see, e.g., \cite[Remark 8.48(d)]{EF-book}.} endowed with the strong$^*$ (or, equivalently, weak or strong) operator topology for the corresponding Koopman operators making it a Polish space (i.e., completely metrizable and separable) by Halmos \cite[pp.~62--64]{Halmos-book}. 
Here for an (invertible measure-preserving) transformation $T$ on $[0,1]$, the corresponding \emph{Koopman operator}\footnote{Following the tradition in ergodic theory, we denote the Koopman operator by the same letter as the underlying transformation. %Note that the Koopman operator can be defined in the same way on $L^p[0,1]$ for every $p\in[1,\infty]$.
}  $T:L^2[0,1]\to L^2[0,1]$ is the unitary operator defined by
$$
(Tf)(x):=f(Tx).
$$ 
 Moreover, we denote by $T_0$ the restriction of a Koopman operator $T$ 
%on $L^2[0,1]$ 
to the closed $T$-invariant subspace 
$$
L_0^2[0,1]:=\left\{f\in L^2[0,1]:\ \int_{[0,1]}f\, dm=0\right\}.
$$ 
The restriction of $\mathcal{T}$ to $L_0^2[0,1]$ will be denoted by $\mathcal{T}_0$.

For $T\in \mathcal{T}$, we are interested in the \emph{weak limit semigroup} $\Tlim$ of $T$ defined by 
\begin{eqnarray*}
\Tlim:=\{S \text{ on } L^2[0,1]:\, \text{weak-}\lim_{k\to\infty}T^{n_k}=S \text{ for some subsequence } (n_k) \text{ of }\N \}.
%(n_k)\subset \N \text{ with }\limk n_k=\infty
%n_k\to\infty
\end{eqnarray*}
%subsequence $(n_k)$ of $\N$}\},

%where the convergence is again considered with respect to the weak operator topology. 
%\begin{rem}
Note that elements of $\Tlim$ do not necessarily belong to $\mathcal{T}$, e.g., for weakly mixing transformations this set contains the orthogonal projection onto the constants which we denote by $\Pconst$. %More generally, by the Jacobs--Glicksberg--de Leeuw decomposition one has elements in  $\Tlim$  which are not in $\mathcal{T}$ (and have non-trivial kernel)  as soon as the Kronecker factor is not the whole $L^2[0,1]$.  %, see, e.g.,  ............. below. %of Koopman operators 
%and $\mathcal{T}_0$ is the set of all operators in $\mathcal{T}$ restricted to $L^2[0,1]$.
% as the set of all $S\in \mathcal{T}$ with $\lim_{k\to\infty}T^{n_k}=S$ in the weak (or, equivalently, strong) operator topology for some subsequence $(n_k)$ of $\N$.
%
%\end{rem}
%We also call the \emph{strong$^*$ limit group} 

Since every operator $S\in\Tlim$ satisfies $S\Eins=\Eins$ and leaves $L_0^2[0,1]$ invariant, the question restricts to understanding the set
$$
\Tnulllim=\{S\text{ on } L_0^2[0,1]:\, \text{weak-}\lim_{k\to\infty}T^{n_k}_0=S \text{ for some subsequence $(n_k)$ of $\N$}\}.
$$

A property of elements of $\mathcal{T}$ is called \emph{generic} (or \emph{typical}) if the set of transformations satisfying this property is comeager in $\mathcal{T}$, i.e, contains a dense $G_\delta$ subset of $\mathcal{T}$. %An important result in the study of generic properties of invertible measure-preserving transformations is Halmos' Conjugacy Lemma, stating that if $\tau \in \mathcal{T}$ is aperiodic, then the conjugacy class $\{ \sigma \tau \sigma^{-1} : \sigma \in \mathcal{T} \} $ is dense in $\mathcal{T}$. 

%We call a flow $(T_t)$ of measure-preserving transformations on a probability space $(X,\mu)$ \emph{continuous} if it has strongly continuous Koopman semigroup, i.e., if for every $f\in L^2[0,1]$ the map $t\mapsto T_tf$, $\R_+\to L^2[0,1]$ is continuous\footnote{In this case,  the Koopman semigroup is strongly continuous in $L^p[0,1]$ for every $p>1$.}.We say that a flow $(T_t)$ \emph{embeds} a transformation $T$ if $T=T_1$. 
%%We furthermore say that two flows $(S_s)$ and $(T_t)$ \emph{intersect} if $S_s=T_t$ holds for some $s,t\in(0,\infty)$. In this case we say that they intersect at time $t$ for the flow $(T_t)$.

%.....TODO: mention polynomials from Ryzhikov etal???.....

\subsection{General Hilbert space operators} \label{introsubsecoperators} Let $H$ be a separable Hilbert space. We denote by $\mathcal{C}$, $\I$ and $\U$ the sets of contractions, isometries and unitary operators on $H$, respectively. We denote by $\wot$, $\sot$ and $\sote$ the weak, the strong and the strong$^*$ operator topology on $\mathcal{C}$. Let us recall that $\wot$ is generated by the semi-norms $\lVert T \rVert_{x,y} = \lvert \la Tx,y \ra \rvert$ for $x, y \in H$, while $\sot$ is generated by the semi-norms $\lVert T \rVert_x = \lVert Tx \rVert$ for $x \in H$. The $\sote$ topology is induced by the semi-norms of the form $\lVert T \rVert_x = \lVert Tx \rVert$ and $\lVert T \rVert_{x,*} = \lVert T^*x \rVert$ for $x \in H$. %The $\sote$ topology corresponds to the $\sot$ topology on the operators and their adjoints. 
Since $H$ is reflexive and separable, all these operator topologies are Polish on $\mathcal{C}$. 
%(i.e., completely metrizable and separable). 
A metric inducing the $\wot$ topology is given, for example, by 
$$
d(T,S):=\sum_{j,l=1}^\infty\frac{|\la (T-S)x_j,x_l\ra|}{2^{j+l}\|x_j\|\|x_l\|}
$$
for a fixed sequence $(x_j)_{j \geq 1}$ which is dense in $H\setminus\{0\}$. 
Similarly, metrics compatible with the $\sot$ and $\sote$ topologies can be defined in the same manner (see also \cite[Proposition 2.2]{EisnerMatrai13} and \cite{EisnerSereny08}).
It is not difficult to see that the topologies $\wot$ and $\sot$ coincide on $\I$ and that the topologies $\wot$, $\sot$ and $\sote$ coincide on $\U$. Moreover, for each $n \geq 1$, the map $T \mapsto T^n$ is $\sot$-continuous on the space $\mathcal{C}$.

In the same way as for transformations, a property of elements of $\mathcal{C}$, $\I$ or $\U$ is said to be \emph{typical} for a given topology on these spaces if the set of operators satisfying this property is comeager in the corresponding space. In this article, we study the weak limit semigroup of a typical Hilbert space operator in the spaces $(\mathcal{C}, \wot)$, $(\I, \sot)$ and $(\U, \sote)$. In particular, we investigate which functions of a typical contraction $T$ of $H$ belong to the weak limit semigroup of $T$. This part of our study will mainly rely on the properties of the weak limit semigroup of a generic transformation in $\mathcal{T}$ as we will see.

\subsection{Positive operators on $L^p[0,1]$}
We also study the weak limit semigroup of a typical positive contraction of $L^p[0,1]$ for $1 \leq p < \infty$. An operator $T$ on $L^p[0,1]$ is said to be \emph{positive} if $Tf \geq 0$ for every function $f \geq 0$ in $L^p[0,1]$. We respectively denote by $\Cp$, $\I_p$ and $\U_p$ the sets of positive contractions, positive isometries and positive invertible isometries of $L^p[0,1]$ for $1 \leq p < \infty$. Examples of positive operators are, of course, Koopman operators on $L^p[0,1]$ associated to transformations. The description of the sets $\I_p$ and $\U_p$ will be recalled in Section \ref{Section-genericposoperators}. 

Analogously, the operator topology $\sot$ can be defined on a separable Banach space and the topologies $\wot$ and $\sote$ can be defined on a reflexive separable Banach space in the same way as in the Hilbert space setting, defining Polish topologies on these spaces. It is not hard to see that the set $\Cp$ is closed in the set of contractions of $L^p[0,1]$ for $\sot$ when $p \geq 1$, as well as for $\wot$ and $\sote$ when $1 < p < \infty$ (see also the proof of Lemma \ref{lemmeadjointposetptscont}). In particular, the spaces $(\Cp, \wot)$, $(\I_p, \sot)$ and $(\U_p, \sot)$\footnote{We use $\sot$ and not $\sot^*$ for positive invertible isometries for convenience to stay in the same $L^p$-space.} are Polish for $1 < p < \infty$, as well as $(\I_1, \sot)$ and $(\U_1, \sot)$. Finally, since $L^p[0,1] $ is uniformly convex for $1 < p < \infty$, the topologies $\wot$ and $\sot$ coincide on $\I_p$ and the topologies $\wot$, $\sot$ and $\sote$ coincide on $\U_p$ (see the paragraph above Theorem \ref{thmtypicalWOTbijisometries}).

\subsection{First properties of the weak limit semigroup}

%We begin with easy observations, 
The following observations on the weak limit semigroup of an operator are easy,
where (d) is due to Ryzhikov \cite[Thm.~6.1]{Ryzhikov07b} and (h) was observed by King \cite[(0.2)]{King86} in the
context of measure-preserving transformations. By considering unitary or normal operators, we automatically assume the underlying space to be Hilbert. Note that we consider  Banach spaces with separable dual to make sure that the weak operator topology is metrizable on bounded subsets of operators. 

\begin{prop}[Properties of the weak limit semigroup]\label{rem:prop-of-limits}
For a contraction $T$ on a Banach space with separable dual,
%Hilbert space,
%unitary operator. 
the following assertions hold.
%For $T\in\mathcal{T}$ denote by $S_l$ the set of operators $V$ on $L^2_0(X,\mu)$ such that there exists a subsequence $(n_k)$ of $\N$ such that  $\lim_{k\to\infty} T_0^{n_k}=V$ weakly. Then the following assertions hold.
\begin{itemize}
\item[(a)] \label{PropPrelim(a)} $\Tlim$ is closed with respect to the weak operator topology, is 
commutative, commutes with $T$ and is closed under multiplication by $T$. 
%(by shifting the corresponding  sequence $(n_k)$ to the right)
If $T$ is invertible then $\Tlim$ is also closed under multiplication by $T^{-1}$.  %Moreover, if $X'$ is separable then $\Tlim$ is closed with respect to the weak operator topology.}%$=T^*$. % (by shifting the corresponding sequence $(n_k)$ to the left). 
\item[(b)] \label{PropPrelim(b)} $\Tlim$ is closed under multiplication. 
%(Indeed, assume that $\lim_{k\to\infty} T^{n_k}=V$ and $\lim_{k\to\infty} T^{m_k}=W$ weakly. Then by (a) $T^{m_k}V\in \Tlim$ for each $k$ and hence $WV\in \Tlim$.)
%
\item[(c)] \label{PropPrelim(c)} Assume that\footnote{Operators satisfying $I\in\Tlim$ are called \emph{rigid}, see Definition \ref{def:rigid-op} below.} $I\in \Tlim$ and that $T$ is unitary. Then $\Tlim^*=\Tlim$. 
%(Indeed, in this case by (a) we have $T^*\in \Tlim$  and hence $V^*\in \Tlim$ for every $V\in \Tlim$.) 
\item[(d)] \label{PropPrelim(d)} % As observed by Ryzhikov \cite[Thm.~6.1]{Ryzhikov07b}, 
The weak limit semigroup $\Tlim$ is convex as soon as  the operators $\frac12(I+T^n)$ belong to $\Tlim$ for every $n\in\N$ large enough. 
%Indeed, assume that $T^{n_k}\to S_1$ and $T^{m_k}\to S_2$ weakly and assume without loss of generality $n_k\geq m_k$ for each $k$. Then by (a) all operators  $\frac12(T^{n_k}+T^{m_k})=T^{m_k}\cdot\frac12(I + T^{n_k-m_k})$ belong to the weak limit semigroup $\Tlim$ and so does their weak limit as $k\to\infty$ being $\frac12(S_1+S_2)$. 
\item[(e)] \label{PropPrelim(e)} If %$T$ is a contraction and  
there exists a sequence $(n_k)\subset \N$ such that $\displaystyle \limk T^{n_k}=I$ weakly %(and hence strongly), 
then 
\begin{equation}\label{eq:wlimsgr=wsgr}
\Tlim%=\ol{\{T^n:\, n\in\Z\}}^{\text{weak}} 
=\ol{\{T^n:\, n\in\N_{0}\}}^{\emph{\wot}}.
\end{equation}
%Indeed, since \eqref{eq:wlimsgr=wsgr} is clearly true for periodic operators, we can assume without loss of generality that $\limk n_k=\infty$, i.e., $I\in \Tlim$. Thus \eqref{eq:wlimsgr=wsgr} follows from (a) and (b).  
If in addition $T$ is invertible, then one also has 
\begin{equation}\label{eq:wlimsgr=wsgr-Z}
\Tlim=\ol{\{T^n:\, n\in\Z\}}^{\emph{\wot}}. 
%=\ol{\{T^n:\, n\in\N_{0}\}}^{\text{weak}}.
\end{equation}
%This is a direct consequence of \eqref{eq:wlimsgr=wsgr} and (c). 
Moreover, for unitary operators \eqref{eq:wlimsgr=wsgr-Z} holds whenever $\displaystyle \limk T^{n_k}=I$ weakly for some sequence $(n_k)\subset \Z$. % since weak convergence to the identity implies weak convergence of the adjoints to the identity.
\item[(f)] \label{PropPrelim(f)} If %I\in \Tlim$ and 
$T$ is a contraction on a Hilbert space and $\Tlim$ contains an injective operator (for example, if $I\in\Tlim$), then $T$ is unitary. %Moreover, 
%This follows from the Foguel decomposition of Hilbert space contractions into a unitary and a weakly stable part, see, e.g., \cite[Thm.~2.3.1]{E-book}. Moreover, both subspaces are reducing for  every $S\in\Tlim$ and the restriction of every $S\in\Tlim$ to the weakly stable subspace of $T$ is zero, thus it suffices to determine $\Tlim$ on the unitary subspace of $T$.
%the question of understanding $\Tlim$ restricts to understanding the restriction of $\Tlim$ to the unitary subspace of $T$. %The same holds if one replaces $I$ by any injective operator.
\item[(g)] \label{PropPrelim(g)} If $T$ is normal, then every $S\in\Tlim$ is normal. %(Indeed, if $\limk T^{n_k}=S$ weakly, then $T^{n_k}T^{*m}=T^{*m}T^{n_k}$ converges weakly as $k\to\infty$ to $ST^{*m}=T^{*m}S$ for every $m\in\N$. For $m\in\{n_k\}$ taking weak limit as $k\to\infty$ implies $SS^*=S^*S$.) 
In particular, if such $S$ is an isometry or a co-isometry, then $S$ is automatically unitary\footnote{cf.~King \cite[p.~363]{King86}} (and $T$ is also unitary %if contraction 
by (f)).
\item[(h)] \label{PropPrelim(h)} If %\footnote{\color{blue}see King \cite[(0.2)]{King86} 
$T$ is unitary and $\Tlim$ contains a (co-)isometric($=$unitary) operator, then $I\in\Tlim$. %(This follows from (1) and the calculation in (g).)
\end{itemize}
\end{prop}
\begin{proof}
(a) Closedness of $\Tlim$ under multiplication with $T$ follows by shifting the corresponding  sequence $(n_k)$ to the right, while, for invertible operators, closedness of $\Tlim$ under multiplication with $T^{-1}$ follows by shifting the corresponding  sequence $(n_k)$ to the left. The rest of the assertion is clear. 

(b) Assume that $\displaystyle \lim_{k\to\infty} T^{n_k}=V$ and $ \displaystyle \lim_{k\to\infty} T^{m_k}=W$ weakly. Then by (a) 
$T^{m_k}V\in \Tlim$ for each $k$ and hence $WV\in \Tlim$.

(c) By (a) we have $T^*\in \Tlim$  and hence $V^*\in \Tlim$ for every $V\in \Tlim$.

(d) 
Let $k_0$ be such that $\frac12(I+T^k)\in\Tlim$ for every $k\geq k_0$. Assume that $T^{n_k}\to S_1$ and $T^{m_k}\to S_2$ weakly. Passing to a subsequence of $(n_k)$ if necessary we can 
 assume without loss of generality that $n_k\geq m_k+k_0$ holds for each $k$. Then by (a) all operators  $\frac12(T^{n_k}+T^{m_k})=T^{m_k}\cdot\frac12(I + T^{n_k-m_k})$ belong to the weak limit semigroup $\Tlim$ and so does their weak limit as $k\to\infty$ being $\frac12(S_1+S_2)$. 

(e) Since \eqref{eq:wlimsgr=wsgr}  clearly holds for periodic operators, we can assume without loss of generality that $\displaystyle \limk n_k=\infty$, i.e., $I\in \Tlim$. Thus both \eqref{eq:wlimsgr=wsgr} and \eqref{eq:wlimsgr=wsgr-Z} follow from (a). The last assertion holds by applying the previous argument to $T^{-1}$ if necessary.
%Since \eqref{eq:wlimsgr=wsgr} is clearly true for periodic operators, we can assume without loss of generality that $\limk n_k=\infty$, i.e., $I\in \Tlim$. Thus \eqref{eq:wlimsgr=wsgr} follows from (a) and (b). Moreover, \eqref{eq:wlimsgr=wsgr-Z} is a direct consequence of \eqref{eq:wlimsgr=wsgr} and (c), whereas the last assertion holds since weak convergence to the identity implies weak convergence of the adjoints to the identity.

(f) This point follows from the Foguel decomposition of the Hilbert space 
%contraction $T$ 
into two invariant subspaces, where $T$ is weakly stable (i.e., $\displaystyle \limn T^n=0$ weakly) on the first subspace and $T$ is unitary on the second subspace,
%a unitary and a weakly stable part and a unitary part, 
see, e.g., \cite[Thm.~II.3.9]{E-book}.

(g) Assume that $T$ is normal. If $\displaystyle \limk T^{n_k}=S$ weakly, then $T^{n_k}T^{*m}=T^{*m}T^{n_k}$ converges weakly as $k\to\infty$ to $ST^{*m}=T^{*m}S$ for every $m\in\N$. For $m\in\{n_k\}$ taking weak limit as $k\to\infty$ implies $SS^*=S^*S$, i.e., $S$ is normal.

(h) Let $S\in \Tlim$. %Then both $S$ and $T$ are unitary by (f) and (g). 
Since by (a) $\Tlim$ is closed under the multiplication by $T^*$, we have $(T^*)^nS,S(T^*)^n\in \Tlim$ for every $n\in\N$, implying that $S^*S, SS^*\in \Tlim$. Thus if $S$ is (co-) isometric (and hence unitary by (g)), $I\in \Tlim$ follows.
\end{proof}

\begin{remark}
Let $T$ be a contraction on a Hilbert space with the corresponding  decomposition into a unitary and a weakly stable part as in the proof of (f). Then  
    both subspaces are $S$-invariant for every $S\in\Tlim$ and the restriction of $S$ to the weakly stable subspace of $T$ is zero.  Thus it suffices to determine $\Tlim$ on the unitary subspace of $T$.
\end{remark}

\begin{remark}\label{rem:isom-strong-limit-sgr}
Let $T$ be an isometry on a Hilbert space. Then an operator of the form $\frac12(I+T^m)$, $m\in\N$, used in Proposition \ref{rem:prop-of-limits}(d) belongs to the \emph{strong limit semigroup} 
$$
\Tlimstr:=\{S \text{ on } H:\, \lim_{k\to\infty}T^{n_k}=S \text{ strongly for some subseq. } (n_k) \text{ of }\N \}
$$
of $T$ if and only if $T^m=I$. Indeed, such operators are isometries (which is a necessary condition to belong to $\Tlimstr$ for an isometry $T$) only if $T^m=I$ as the following argument shows. Assume that $\frac12(I+S)$ is an isometry for some isometry $S$ and let $f\in H\setminus\{0\}$. By assumption we have $\|f+Sf\|^2=4\|f\|^2$ or, equivalently, 
\begin{equation}\label{eq:CS-equality}
\Re \la Sf,f\ra=\|f\|^2.
\end{equation}
Equality in the Cauchy-Schwarz inequality leads to $Sf=\lambda f$ for some $\lambda\in \T$ (recall that $S$ is an isometry). This together with \eqref{eq:CS-equality} implies $\lambda=1$, so $S=I$.
\end{remark}

\section{$\lambda$-Rigidity}\label{sec:rig}

\subsection{Fixed asymptotics and $\lambda$-rigidity: transformations}
%in the weak closure 
%$\lambda$-rigidity for transformations}

We now show that, in order to study $\Ttnull$, one has to allow the operators inside to depend on $T$. We first need an easy classical fact about commuting operators.

\begin{lem}\label{lem:comm}
Let $T$ and $S$ be commuting linear operators on a Banach space $H$ and let $f\in H$ satisfy $Tf=\lambda f$ for some $\lambda\in\C$. Then $T(Sf)=\lambda Sf$. 
\end{lem}
\begin{proof}
The assertion follows by $TSf=STf=\lambda Sf$. 
\end{proof}

%{\color{red}We now show that $\mathcal{M}$ can be dense (and hence a generic transformation belongs to $\mathcal{M}$) only for operators $V$ of the form $\lambda I$.}
%The following is an ergodic theoretic counterpart of \cite[Prop.~IV.3.14]{E-book}.

\begin{prop}\label{thm:lambdaI}
Let $V$ be a linear operator on $L_0^2[0,1]$. Assume that the set
% $\mathcal{M}$
$$
\mathcal{M}:= \{T\in \mathcal{T}:\ \limk T^{n_k}_0= V \text{ weakly for some subsequence $(n_k)$ of $\N$}\}
$$
is dense in $\mathcal{T}$. Then $V=\lambda I$ for some $\lambda\in [0,1]$.
%$\lambda\in\D$. 
\end{prop}
%In particular, a generic transformation $T$ can have only operators of the form $\lambda I$ with $\lambda\in\D$ as weak limit points of the powers of $T_0$. This gives an ergodic theoretic counterpart of \cite[Prop.~IV.3.14]{E-book}.
\begin{proof}
We first show that $V=\lambda I$ holds for some $\lambda\in\C$. For this we first observe that $T_0V=VT_0$ holds for every $T\in \mathcal{T}$. This goes as in \cite[Prop.~IV.3.14]{E-book} and follows from the fact that $V$ commutes with $T_0$ for every $T\in \mathcal{M}$ by $T_0V=\displaystyle \lim_{k\to\infty} T_0^{n_k+1}=VT_0$ and the density of $\mathcal{M}$ in $\mathcal{T}$. 

Denote by $f_k$ the functions on $[0,1]$ given by $f_k(x):=e^{2\pi i k x}$ for $k\in \Z$. Note that $(f_k)_{k\neq 0}$ is an orthonormal basis of $L_0^2[0,1]$. Since by the above $V$ commutes with the operator $R_0$ for an irrational rotation $R$ on $[0,1]$ and each $f_k$ is an eigenfunction of $R$, Lemma \ref{lem:comm} implies that $Vf_k$ is again an eigenfunction of $R$. Since irrational rotations are ergodic and hence have one-dimensional eigenspaces, we see that $Vf_k=\lambda_kf_k$ for some $\lambda_k\in \C$. We thus need to show that all $\lambda_k$ are equal.  

Again by Lemma \ref{lem:comm} we see that for every $S\in \mathcal{T}$ one has $V(S_0f_k)=\lambda_k S_0 f_k$. Assume that for some $k\neq 0$ the function $Sf_k=S_0f_k$ has nontrivial correlation with $f_j$ for every $j$ in some set $M_k\subset \Z\setminus\{0\}$, i.e., $Sf_k=\sum_{j\neq 0} c_j f_j$ with $c_j\neq 0$ for every $j\in M_k$. We show that in this case the eigenvalues $\lambda_j$ are all equal to $\lambda_k$ for $j\in M_k$. Indeed, we have 
$$
VSf_k=V\left(\sum_{j\neq 0} c_j f_j\right)=\sum_{j\neq 0} c_j\lambda_j f_j
$$
and on the other hand 
$$
VSf_k=\lambda_k Sf_k = \sum_{j\neq 0} c_j\lambda_k f_j.
$$
By the uniqueness of the Fourier coefficients and the assumption that $c_j\neq 0$ for every $j\in M_k$ we see that $\lambda_k=\lambda_j$ for every $j\in M_k$. 
It thus suffices to find an invertible measure-preserving transformation $S$ such that for, say, $f_1$ and $f_2$ the corresponding sets of indices satisfy 
\begin{equation}\label{eq:M1M2}
M_1\cup M_2=\Z\setminus \{0\} \quad\text{ and }\quad M_1\cap M_2\neq \emptyset.
\end{equation}

Consider now the transformation $S$ on $[0,1]$ given by 
$$
Sx=\begin{cases}
x, \quad & x\in [0,1/2),\\
\frac32-x, \quad & x\in (1/2,1]
\end{cases}
$$
which is clearly invertible and measure preserving. Compute 
\begin{eqnarray}
\langle Sf_1,f_j\rangle &=&\int_0^{1/2} e^{2\pi i (1-j)x}\,dx +\int_{1/2}^1 e^{2\pi i (\frac{3}2-x) }e^{-2\pi i j x}\,dx\nonumber\\ 
&=&\int_0^{1/2} e^{2\pi i (1-j)x}\,dx -\int_{1/2}^1 e^{2\pi i (-1-j)x}\,dx.\label{eq:Sf1fj}
\end{eqnarray}
For $j\notin\{-1,1\}$  we continue as 
\begin{eqnarray*}
\langle Sf_1,f_j\rangle&=&\frac{e^{\pi i (1-j)}-1}{2\pi i (1-j)} - \frac{1-e^{\pi i (-1-j)}}{2\pi i (-1-j)}
=\frac{-e^{-\pi i j}-1}{2\pi i (1-j)} - \frac{1+e^{-\pi i j}}{2\pi i (-1-j)}\\
&=& \frac{1+e^{-\pi i j}}{2\pi i}\left( \frac1{j-1}+\frac1{j+1}\right)
=\frac{2j(1+e^{-\pi i j})}{2\pi i(j^2-1)}
\end{eqnarray*}
which is nonzero if and only if $j\in 2\Z \setminus \{0\}$. Moreover,  \eqref{eq:Sf1fj} implies 
$$
\langle Sf_1,f_{-1}\rangle=\int_0^{1/2} e^{4\pi i x}\,dx - \frac12=-\frac12
$$
and analogously $\langle Sf_1,f_{1}\rangle=\frac12$. Therefore 
\begin{equation}\label{eq:M1}
M_1=2\Z\setminus \{0\} \cup \{\pm 1\}.
\end{equation}

Compute now 
\begin{eqnarray}
\langle Sf_2,f_j\rangle &=&\int_0^{1/2} e^{2\pi i (2-j)x}\,dx +\int_{1/2}^1 e^{4\pi i (\frac{3}2-x) }e^{-2\pi i j x}\,dx\nonumber\\ 
&=&\int_0^{1/2} e^{2\pi i (2-j)x}\,dx +\int_{1/2}^1 e^{2\pi i (-2-j)x}\,dx.\label{eq:Sf2fj}
\end{eqnarray}
For $j\notin\{-2,2\}$  we continue as 
\begin{eqnarray*}
\langle Sf_2,f_j\rangle&=&\frac{e^{\pi i (2-j)}-1}{2\pi i (2-j)} + \frac{1-e^{\pi i (-2-j)}}{2\pi i (-2-j)}
=\frac{e^{-\pi i j}-1}{2\pi i (2-j)} + \frac{1-e^{-\pi i j}}{2\pi i (-2-j)}\\
&=& \frac{1-e^{-\pi i j}}{2\pi i}\left( \frac1{j-2}-\frac1{j+2}\right)
=\frac{4(1-e^{-\pi i j})}{2\pi i(j^2-4)}
\end{eqnarray*}
which is nonzero if and only if $j\in 2\Z +1$. Moreover, by \eqref{eq:Sf2fj} 
$$
\langle Sf_2,f_{-2}\rangle=\frac12=\langle Sf_2,f_{2}\rangle
$$
implying
\begin{equation}\label{eq:M2}
M_2=(2\Z+1) \cup \{\pm 2\}.
\end{equation}
This we see by \eqref{eq:M1} and \eqref{eq:M2} that the condition \eqref{eq:M1M2} is satisfied.  This proves $V=\lambda I$ for some $\lambda \in \C$.

It remains to show that $\lambda \in [0,1]$. We argue as in Stepin \cite[p.~171]{Stepin86}. Since $V$ is a contraction and maps real functions to real functions by assumption, we have $\lambda\in[-1,1]$. It remains to show $\lambda\geq 0$. Take $\veps\in (0,1)$ and let $f$ be the characteristic function of a set $A$ with $m(A)=\veps$. Write $f=f_1+f_0$ with $f_1$ being the constant function equal to $\int_{[0,1]}f dm=\veps$ and $f_0=f-f_1\in L^2_0[0,1]$. Then
$$
T^{n_k}f\to f_1+\lambda (f-f_1) \quad \text{weakly}.
$$
Since $T$ preserves positivity, we see that the right hand side is greater than or equal to zero. Taking $x\in A$ we obtain $\veps +\lambda (1-\veps)\geq 0.$
Since this holds for every $\veps\in(0,1)$, $\lambda\geq 0$ follows.
% almost everywhere. 
%, we see that $|\lambda|\leq 1$. Moreover, $V$ maps real functions to real functions by assum}
%The proof is complete.}
\end{proof}

\begin{defi}\label{def:lambda-rigid-et}
For $\lambda\in [0,1]$, one calls a measure-preserving transformation $T$ on a probability space $(X,\mu)$ \emph{$\lambda$-rigid} (or \emph{$(1-\lambda)$-mixing}) if there exists a subsequence $(n_k)$ of $\N$ such that $\displaystyle \limk T_0^{n_k}=\lambda I$ in the weak operator topology. A transformation is called \emph{rigid} if it is $1$-rigid, and we call it \emph{$[0,1]$-rigid} if it is $\lambda$-rigid for every $\lambda \in [0,1]$.
%$\lambda I$ belongs to the closure of $\{T_0, T^2_0,\ldots\}$ in the weak operator topology. 
\end{defi}
%Note that this is equivalent to the existence of a subsequence $(n_k)$ of $\N$ with $\lim_{k\to\infty}T^{n_k}_0=\lambda I$ in the weak operator topology. 
Note %further 
that %$1$-rigidity  coincides with rigidity and 
$0$-rigidity coincides with weak mixing. 

The following shows that operators of the form $\lambda I$ for $\lambda\in[0,1]$ indeed appear as elements of $\Tnulllim$ for a typical transformation. 

\begin{thm}[Stepin \cite{Stepin86}]\label{thm:stepin}
For %every Lebesgue probability space $(X,\mu)$ and 
every $\lambda\in[0,1]$, the set of all $\lambda$-rigid transformations 
is a dense $G_\delta$ subset of $\mathcal{T}$. 
\end{thm}
Thus, by taking rational $\lambda$ and intersecting the corresponding dense $G_\delta$ sets, one has the following.
\begin{cor}\label{cor:lambda-rigid-typical}
%A generic transformation in $\mathcal{T}$ is $[0,1]$-rigid.
The set of $[0,1]$-rigid transformations is dense $G_\delta$ in $\mathcal{T}$.
\end{cor}
%{\color{blue}We will present a class of examples of $\overline{\D}$-rigid Gaussian transformations in Section \ref{sec:ex-transf} below. }

 Proposition \ref{thm:lambdaI} and Theorem \ref{thm:stepin} imply the following. Recall that $\Pconst$ denotes the orthogonal projection onto the constants.

\begin{thm}[Fixed asymptotics of generic transformations]
An operator $V\in L^2[0,1]$ belongs to the weak limit semigroup $\Tlim$ for a generic transformation in $\mathcal{T}$ if and only if $V=\lambda I +(1-\lambda)\Pconst$ for some $\lambda\in[0,1]$.
%An operator $V\in L_0^2[0,1]$ belongs to the weak limit semigroup $\Tnulllim$ for a generic transformation in $\mathcal{T}$ if and only if $V=\lambda I$ for some $\lambda\in[0,1]$. 
\end{thm}
%\begin{proof}
%The ``only if'' direction follows from Theorem \ref{thm:lambdaI}. For the ``if'' direction, take a rational $\lambda\in[0,1]$. Stepin \cite{Stepin86} showed that $\lambda I$ is a weak limit of a the powers of a generic measure-preserving transformation. Since a countable intersection of dense $G_\delta$ subsets of $\mathcal{T}$ is again dense $G_\delta$ and the set of weak limit points is closed, a generic transformation has $\lambda I$ as  weak limit of the powers for every $\lambda\in [0,1]$.  
%\end{proof}

\subsection{Fixed asymptotics and $\lambda$-rigidity: operators}
%$\lambda$-rigidity for operators}

For the following operator theoretic counterpart of Proposition \ref{thm:lambdaI}, see \cite[Prop.~IV.3.14 and its proof]{E-book} \footnote{Note that here we use the fact that $\I$ and $\U$ are dense in $\mathcal{C}$ with respect to $\wot$, see \cite{EisnerSereny08}.}. 
%Here, for a Hilbert space $H$ we denote by $\mathcal{L}(H)$ the space of all linear bounded operators on $H$. Moreover, $\mathcal{U}$, $\mathcal{I}$ and $\mathcal{C}$ denote the spaces of unitaries, isometries and contractions on $H$ endowed with the strong($=$ weak), strong$^*$ and weak operator topologies, respectively. {\color{red}///Hier oder früher?///}

\begin{prop}\label{prop-multiple-identity-op}
Let $H$ be a separable Hilbert space and $V\in \mathcal{L}(H)$. If the set of operators $T$ on $H$ satisfying $V\in\Tlim$
%$
%\displaystyle \lim_{k\to\infty} T^{n_k}=V \text{ weakly} 
%$
is dense in one of the spaces $(\mathcal{U}, \emph{\sote})$, $(\mathcal{I}, \emph{\sot})$ and $(\mathcal{C}, \emph{\wot})$, then $V$ is a multiple of the identity. 
\end{prop}
%\begin{remark}\label{rem:weak-implies-strong}
%Since weak operator convergence of contractions to an isometry implies strong operator convergence, see, i.e., \cite[Lemma II.5.11]{E-book}, one can replace ``weakly'' by ``strongly'' in Proposition \ref{prop-multiple-identity-op} for the spaces $\mathcal{U}$ and $\mathcal{I}$ whenever $V$ is an isometry.
%\end{remark}

We now present the operator theoretic counterpart of Definition \ref{def:lambda-rigid-et} and Corollary \ref{cor:lambda-rigid-typical}, and show that operators of the form $\lambda I$ with $\lambda\in \overline{\D}$  do appear as weak %(resp., strong) %strong/weak 
limits for typical operators.
%shows that, as in the case of transformations, $\lambda$-rigidity is typical for operators. 
\begin{defi}\label{def:rigid-op}
Let $\lambda\in \C$. % \in \D:=\{\lambda\in\C:\, |\lambda|\leq 1\}$. A 
A linear bounded operator $T$ on a Hilbert space $H$ is called \emph{$\lambda$-rigid} if there exists a subsequence $(n_k)$ of $\N$ such that $\displaystyle \limk T^{n_k}=\lambda I$  weakly. We call an operator \emph{rigid} if it is $1$-rigid and \emph{$\overline{\D}$-rigid} if it is $\lambda$-rigid for every $\lambda\in \overline{\D}$.
\end{defi}
We are interested  in the case of contractions where %as in Remark \ref{rem:weak-implies-strong}, 
one can replace ``weakly'' by ``strongly'' in the above definition whenever $|\lambda|=1$. (Recall that $\wot$ convergence of contractions to an isometry on a Hilbert space implies $\sot$ convergence, see, e.g., \cite[Lemma II.5.11]{E-book}.)

The following is an operator theoretic analogue of Halmos' Conjugacy Lemma (Proposition \ref{prop:Halmos-conjugacy}) due to Choksi and Nadkarni \cite{ChoksiNadkarni90}, see also Nadkarni \cite[Prop. 8.23]{Nadkarni-book}. 
\begin{proposition}[Conjugacy Lemma for unitary operators]\label{prop:conjugacy-lemma-op}
Let $R$ be a unitary operator on a separable Hilbert space with $\sigma(R)=\T$.
%the closed support of the maximal spectral type of $R$ equals $\T$. 
Then the conjugacy class $\{S^{-1}RS:\, S\in\mathcal{U}\}$ is dense in $(\mathcal{U}, \emph{\sote})$.  
\end{proposition}

The next result is a strengthening of \cite[Thm.~IV.3.11]{E-book}, where $\T$-rigidity was handled, cf.~\cite[Remark IV.3.16]{E-book}. Using it, we will show a more general statement in Theorem \ref{thm:generic-op} below.
\begin{thm}[{$\overline{\D}$-rigidity is typical}]\label{thm:D-rigid-residual}
Let $H$ be a separable Hilbert space. Then $\overline{\D}$-rigid operators on $H$ form a dense $G_\delta$ set in the spaces $(\mathcal{U}, \emph{\sote})$ and $(\mathcal{I}, \emph{\sot})$, and a comeager\footnote{The presented proof of the $G_\delta$-property does not work for $\mathcal{C}$ because,  in opposite to strong convergence, weak convergence of operators does not imply weak convergence of their powers.} set in the space $(\mathcal{C}, \emph{\wot})$.
\end{thm}
For a class of examples of operators with a much stronger property than $\overline{\D}$-rigidity, see Section \ref{sec:ex-op} below. 
\begin{proof}
We first show the assertion for $\mathcal{U}$ and $\mathcal{I}$. 
It clearly suffices to show the dense $G_\delta$ property of $\lambda$-rigid operators for a fixed $\lambda\in \D$ by taking $\lambda$'s with rational coordinates and intersecting the corresponding dense $G_\delta$ sets. 

The $G_\delta$-property in $\I$ and $\U$ follows from the representation of the set under consideration as 
$$
\bigcap_{N\in \N} \bigcup_{n\geq N} \left \{T:\  d(T^n,\lambda I) <\frac1N\right \},
$$
where $d$ denotes the metric inducing the $\wot$ topology on the space of contractions on $H$ (see Subsection \ref{introsubsecoperators}).

Since $\mathcal{U}$ is dense in $\mathcal{I}$ by \cite{E10}, it suffices to  show the density property in $\mathcal{U}$.  Consider first $\lambda\in [0,1]$. Since all separable Hilbert spaces are isomorphic, we can assume without loss of generality that $H=L_0^2[0,1]$.
Let $U\in\mathcal{U}$, $\veps>0$, $d\in\N$ and $x_1,\ldots,x_d\in H$. 
Denote by $R$ the Koopman operator of the translation on $[0,1]$ by an irrational number. By Proposition \ref{prop:conjugacy-lemma-op}, the conjugacy class of $R_0$ is dense in $\mathcal{U}$. In particular, there exists $S\in \mathcal{U}$ such that $\|S^{-1}R_0Sx_j-Ux_j\|<\veps$ for every $j\in\{1,\ldots,d\}$. Since $\lambda$-rigid transformations are dense in $\mathcal{T}$ by Theorem \ref{thm:stepin}, there exists a $\lambda$-rigid unitary operator $T$  on $H$ with the property
$\|(T-R_0)Sx_j\|<\veps$ for every $j\in\{1,\ldots,d\}$, implying that
\begin{align*}
    \|S^{-1}TSx_j-Ux_j\| &\leq \|S^{-1}(T-R_0)Sx_j\|+\|S^{-1}R_0Sx_j-Ux_j\| 
    \leq 2 \veps
\end{align*}
for every $j\in\{1,\ldots,d\}$.
%$$
%\|S^{-1}TSx_j-Ux_j\|\leq \|S^{-1}(T-R_0)Sx_j\|+\|S^{-1}R_0Sx_j-Ux_j\|\leq 2\veps \quad \forall j\in\{1,\ldots,d\}.
%$$
Since the operator $S^{-1}TS$ is $\lambda$-rigid as well, the density of $\lambda$-rigid operators in $\mathcal{U}$ follows. 

For an arbitrary $\lambda\in \overline{\D}$ write $\lambda=re^{2\pi i \varphi}$ with $r,\varphi\in[0,1]$. Let $U$ be a unitary operator on $H$. By the above, for every strong neighbourhood $\mathcal{V}$ of $U$ there exists a unitary operator $T\in\mathcal{V}$ and a subsequence $(n_k)$ of $\N$ such that $\displaystyle \lim_{k\to\infty}T^{n_k}=rI$ in the weak operator topology. By Kuipers, Niederreiter \cite[Thm. 4.1]{KuipersNiederreiter}, for a.e.~$\psi\in [0,1]$  the sequence $(e^{2\pi i n_k\psi})_{k \geq 1}$ is equidistributed in $\T$. In particular, there exists a subsequence $(n'_k)$ of $(n_k)$ such that $\displaystyle \lim_{k\to\infty}e^{2\pi i n'_k\psi}=e^{2\pi i\varphi}$. For such $\psi$ consider the unitary operator $e^{2\pi i \psi}U$. We have 
$$
\lim_{k\to\infty}(e^{2\pi i \psi}T)^{n'_k}=\lim_{k\to\infty}e^{2\pi i n'_k \psi}T^{n'_k}=e^{2\pi i\varphi}rI=\lambda I
$$
in the weak operator topology, i.e., for each such $\psi$ the operator $e^{2\pi i \psi}T$ is $\lambda$-rigid. Since this holds for a.e.~$\psi$, there are such $\psi$ so close to $0$ that $e^{2\pi i \psi}T\in \mathcal{V}$. The proof of the density statement for $\mathcal{U}$ is complete. 

Since $\mathcal{U}$ is comeager in $\mathcal{C}$ by \cite{E10} (see also Theorem \ref{thm:unitary-G-delta-in-I-C}) and the $\wot$ and $\sote$ topologies coincide on $\mathcal{U}$, every set which is comeager in $\mathcal{U}$ is also comeager in $\mathcal{C}$. This shows the assertion for  $\mathcal{C}$.
\end{proof}

\subsection{Fixed asymptotics and $\lambda$-rigidity: positive operators} 
We now focus on positive operators on $L^p[0,1]$ with $1 < p < \infty$. We have the following analogue of Proposition \ref{prop-multiple-identity-op} in this context.

\begin{proposition} \label{fixedasymptoticposoperators}
    Let $1 < p < \infty$ and let $V$ be a positive operator on $L^p[0,1]$. If the set of positive operators of $L^p[0,1]$ satisfying $V\in\Tlim$ is dense in one of the spaces $(\Cp, \emph{\wot})$, $(\I_p, \emph{\sot})$ and $(\Up, \emph{\sot})$, then $V = \lambda I$ for some $\lambda \in [0,1]$.
\end{proposition}

\begin{proof}
    Since $\I_p$ and $\U_p$ are $\wot$-dense in $(\Cp, \wot)$ (see Theorem \ref{thWOTdenseposisometriesbijgralew} below), we just have to prove Proposition \ref{fixedasymptoticposoperators} for positive contractions. If the corresponding set is dense in $\Cp$, then by rescaling $V$ must commute with every positive operator of $L^p[0,1]$. In particular, $V$ commutes with every positive operators $T : f \mapsto \int_{[0,1]} f \, dm \cdot e$ for $e \geq 0$ and $e \ne 0$. It implies that for every $e \geq 0$ with $e \ne 0$, there exists $\lambda_e \geq 0$ such that $V e = \lambda_e \, e$. Indeed, we can take $\lambda_e = \int_{[0,1]} Ve \, dm/ \int_{[0,1]} e \, dm$ using the fact that $\int_{[0,1]} e \, dm >0$ since $e$ is positive and non-zero. It is easy to see by considering linearly independent functions that there exists $\lambda \geq 0$ such that $V e = \lambda \, e$ for every $e \geq 0$. 
    This last relation easily extends to every $e \in L^p[0,1]$ by considering the real and imaginary parts of $e$, as well as their positive and negative parts. This concludes the proof of Proposition \ref{fixedasymptoticposoperators}.
\end{proof}

We will also see in Theorem \ref{typweaklimsemigpconvexposcontractionsWOT} that the converse of Proposition \ref{fixedasymptoticposoperators} holds.

\section{A class of examples: operators}\label{sec:ex-op}

The following beautiful example inspired by the construction of Bayart and Matheron \cite[Section 3.4]{BayartMatheron06} was communicated to us by Sophie Grivaux. Recall that a compact subset $E$ of $\T$ is called a \emph{Kronecker set} if the characters on $\T$ are uniformly dense in the set of all $\T$-valued continuous functions on $E$. For a construction of a perfect\footnote{Perfect sets always support continuous probability measures.} % since they contain a continuous image of a Cantor set.}}
Kronecker set see, e.g., \cite[Chapter 5]{Rudin-book}.
\begin{prop} \label{prop:ex-Sophie}
For every continuous probability measure $\mu$ on $\T$ supported on a Kronecker set $E$, the multiplication operator $M_z$ on $L^2(\T,\mu)$ given by $(M_zf)(z):=zf(z)$ satisfies the following property: the weak limit semigroup of $M_z$ coincides with the set of all contractions on $L^2(\T,\mu)$ commuting with $M_z$, i.e., with the set of all multiplicators on $L^2(\T,\mu)$ by a function with $L^\infty$-norm  $\leq 1$.
\end{prop}
\begin{remark}\label{rem:Kronecker-sets}
\begin{itemize}
    \item[(a)] 
Since $E$ is a Kronecker set, the function $\indic$ is a uniform limit on $E$ of a sequence $z^{n_k}$ with $n_k \in \Z$ and $\lvert n_k \rvert \to \infty$. Since the function $\indic$ is real valued, one has that $\lVert \indic - z^{n_k} \rVert_\infty = \lVert \indic - z^{-n_k} \rVert_\infty$, so we may assume that some subsequence of $(n_k)$ tends to $\infty$. In particular, it implies that the weak closure of the powers $M_z^n$ with $n \in \Z$ coincides with the weak limit semigroup of $M_z$ by Proposition \ref{rem:prop-of-limits}(e).
\item[(b)]
The operators from Proposition \ref{prop:ex-Sophie} are in particular $\overline{\D}$-rigid.
\end{itemize}
\end{remark}

The inclusion ``$\subset$'' in Proposition \ref{prop:ex-Sophie} is clear. For the converse inclusion, we will need the following lemma. This lemma is the analog of Bayart, Matheron \cite[Claim 3.8]{BayartMatheron06} and relies on Lyapunov's convexity theorem, which states that for every non-atomic probability space $(\Omega, \mu)$ and for every measurable set $A \subset \Omega$, the range of $\mu_{|A}$ is the whole interval $[0, \mu(A)]$ (see, for instance, \cite{Lyapunov40}). 

\begin{lemma} \label{LemmepreuveexSophie}
Let $f_1, \dotsc, f_r, g_1, \dotsc, g_r \in L^\infty(\T,\mu)$, let $f \in L^\infty(\T,\mu)$ such that $\lVert f \rVert_\infty \leq 1$ and let $\gamma > 0$. Then there exists a function $h \in L^\infty(\T,\mu) $ supported on $E$ such that $\lvert h \rvert = 1$ on $E$ and
\begin{align*}
    \left \lvert \int_E f f_i \overline{g_i} d\mu -  \int_E h f_i \overline{g_i} d\mu  \right \rvert < \gamma, \; \forall 1 \leq i \leq r.
\end{align*}
\end{lemma}

We start by proving Proposition \ref{prop:ex-Sophie} assuming Lemma \ref{LemmepreuveexSophie}.
\begin{proof}[{Proof of Proposition \ref{prop:ex-Sophie}}] 
Let $V$ be a contraction on $L^2(\T,\mu)$ commuting with $U := M_z$. Then there exists a function $f \in L^\infty(\T,\mu)$ such that $V = M_f$ and $\lVert f \rVert_\infty \leq 1$, where $M_f$ is the multiplication operator by the function $f$. Indeed, if $V$ is a contraction on $L^2(\T,\mu)$ commuting with $M_z$, then $V^*$ also commutes with $M_z$ since $M_z$ is unitary, and thus $V$ commutes with $\varphi(M_z) = M_\varphi$ for every bounded Borel function $\varphi$ on $L^\infty(\T,\mu)$ by the Borel functional calculus. Thus for every Borel subset $B$ of $E$ and for every function $g \in L^2(\T,\mu)$, we have $V(\indic_B \cdot g) = \indic_B \cdot Vg$, implying that there exists a function $f \in L^\infty(\T, \mu)$ with $\lVert f \rVert_\infty \leq 1$ such that $V = M_f$ (see Nadkarni \cite[proof of Lemma 1.25]{Nadkarni-book}). 

We will show that $V$ belongs to the weak closure of the powers $U^n$ with $n \in \Z$, i.e, that for every $\varepsilon > 0$ and for every $f_1, \dotsc, f_r, g_1, \dotsc, g_r \in L^2(\T,\mu)$, there exists an integer $n \in \Z$ such that
\begin{align} \label{equation1preuvesophieex}
    \lvert \langle f f_i, g_i \rangle - \langle U^n f_i, g_i \rangle \rvert < \varepsilon, \; \forall 1 \leq  i \leq r.
\end{align}
Let $\varepsilon > 0$ and let $f_1, \dotsc, f_r, g_1, \dotsc, g_r \in L^2(\T,\mu)$. Since $L^\infty(\T,\mu)$ is dense in $L^2(\T,\mu)$, we can suppose that all the $f_i$ and the $g_i$ belong to $L^\infty(\T,\mu)$. Let us notice that (\ref{equation1preuvesophieex}) corresponds to
\begin{align*}
    \left \lvert \int_E f(z) f_i(z) \overline{g_i(z)} d\mu(z) - \int_E z^n f_i(z) \overline{g_i(z)} d\mu(z) \right \rvert < \varepsilon, \; \forall 1 \leq i \leq r.
\end{align*}
%Let $0 <\rho < 1 $ such that
%\begin{align}\label{equation2preuvesophieex}
%    \lvert \langle \rho f_i f, g_i \rangle - \langle f f_i, g_i \rangle \rvert < %\varepsilon/2, \; \forall 1 \leq i \leq r.
%\end{align}
%If we find an integer $n \in \Z$ such that
%\begin{align} \label{equation3preuvesophieex}
%    \lvert \langle \rho f f_i, g_i \rangle - \langle z^n f_i, g_i \rangle \rvert %< \varepsilon/2, \; \forall 1 \leq i \leq r,
%\end{align}
%then (\ref{equation1preuvesophieex}) will hold. 
%Since $\lVert \rho f \rVert_\infty = \rho < 1$, we can suppose that $\lVert f \rVert_\infty < 1$ and it is enough to show that (\ref{equation3preuvesophieex}) holds for $f$ instead of $\rho f$. 

Let $h \in L^\infty(\T,\mu)$ be given by Lemma \ref{LemmepreuveexSophie} for $\gamma = \varepsilon/2$ and for the functions $f$, $f_i$ and $g_i$, and let $\delta > 0$ be such that $\delta \displaystyle \max_{1 \leq i \leq r} \lVert f_i g_i \rVert_\infty < \varepsilon/2 $. Since $E$ is a Kronecker set, there exists an integer $n \in \Z$ such that $\lVert z^n - h \rVert_\infty < \delta$. Thus, for every $1 \leq i \leq r$,
\begin{align*}
    \left \lvert \int_E z^n f_i \overline{g_i} d\mu - \int_E h f_i \overline{g_i} d\mu \right \rvert < \delta \lVert f_i g_i \rVert_\infty < \varepsilon/2.
\end{align*}
%since $f_i, g_i \in L^\infty(\T,\mu)$. 
We thus obtain that 
\begin{align*}
    \left \lvert \int_E z^n f_i \overline{g_i} d\mu - \int_E f f_i \overline{g_i} d\mu \right \rvert < \varepsilon
\end{align*}
for every $1 \leq i \leq r$, implying (\ref{equation1preuvesophieex}). This together with Remark \ref{rem:Kronecker-sets}(a) concludes the proof of Proposition \ref{prop:ex-Sophie}.
\end{proof}

We now prove Lemma \ref{LemmepreuveexSophie}.
\begin{proof}[{Proof of Lemma \ref{LemmepreuveexSophie}}]
Let $h_i := f_i \overline{g_i}$ for every $1 \leq i \leq r$. There exists a partition $(A_1, \dotsc, A_d)$ of $E$ in Borel subsets with $\mu(A_s) > 0$ for every $1 \leq s \leq d$ and sequences of complex numbers $(c_1^i, \dotsc, c_d^i)$ such that 
\begin{align}
    \lVert h_i - \displaystyle \sum_{s=1}^d c_s^i \indic_{A_s} \lVert_1 < \gamma/2, \; \forall 1 \leq i \leq r. 
\end{align}
If we manage to construct a function $h \in L^\infty(\T,\mu)$ supported on $E$ with $\lvert h \rvert = 1$ on $E$ such that $\int_{A_s} f d\mu = \int_{A_s} h d\mu$ for every $1 \leq s \leq d$, then we will obtain that
\begin{align} \label{ineq1preuveexsophie}
    \left \lvert \int_E (f-h) h_i d\mu \right \rvert \leq &\left \lvert \int_E (f-h) \displaystyle \sum_{s=1}^d c_s^i \indic_{A_s} d\mu \right \rvert \\ \notag
    + &\left \lvert \int_E (f-h) (h_i - \displaystyle \sum_{s=1}^d c_s^i \indic_{A_s}) d\mu \right \rvert 
\end{align}
for every $1 \leq i \leq r$. For such a function $h$, the second term on the right hand side of (\ref{ineq1preuveexsophie}) is bounded from above by $\gamma \lVert f-h \rVert_\infty/2$, and thus by $\gamma$. Moreover, the first term in the right hand side of (\ref{ineq1preuveexsophie}) is equal to $\left \lvert \displaystyle \sum_{s=1}^{d} c_s^i \int_{A_s} (f-h) d\mu \right \rvert = 0$ by the definition of the function $h$, and thus we will obtain that $\left \lvert \int_E f h_i d\mu - \int_E h h_i d\mu \right \rvert < \gamma$ for every $1 \leq i \leq r$.

In order to construct the function $h$, we set $\alpha_s := \int_{A_s} f d\mu$. Then $\displaystyle \max_{1 \leq s \leq d} \frac{\lvert \alpha _s \rvert}{\mu(A_s)} \leq \lVert f \rVert_\infty \leq 1$. Thus, there exist $\delta_1, \dotsc, \delta_d > 0$ and $K_1, \dotsc, K_d$ such that $\lvert K_s \rvert = 1$, $0 \leq \delta_s \leq \mu(A_s)$ and $\alpha_s = K_s \delta_s$ for every $1 \leq s \leq d$. Since $\mu$ is a continuous probability measure, Lyapunov's convexity theorem implies that each set $A_s$ can be written as a disjoint union $A_s = \Tilde{A_s} \cup B_s^+ \cup B_s^-$ of Borel subsets with $\mu(\Tilde{A_s}) = \delta_s$ and $\mu(B_s^+) = \mu(B_s^-)$ for every $1 \leq s \leq d$. Indeed, we first choose $\Tilde{A_i} \subset A_i$ such that $\mu(\Tilde{A_i}) = \delta_i$, and then $B_i^+ \subset A_i \setminus \Tilde{A_i}$ such that $\mu(B_i^+) = \frac{1}{2}(\mu(A_i) - \delta_i)$.
We define the function $h$ on $E$ by $h = K_s$ on $\Tilde{A_s}$, $h = 1$ on $B_s^+$ and $h = -1$ on $B_s^-$ for every $1 \leq s \leq d$. Thus $\lvert h \rvert = 1$ on $E$, and $\int_{A_s} h d\mu = K_s \mu(\Tilde{A_s}) + \mu(B_s^+) - \mu(B_s^-) = K_s \delta_s = \alpha_s = \int_{A_s} f d\mu$ for every $1 \leq s \leq d$, as desired.
\end{proof}

\section{A class of examples: transformations}\label{sec:ex-transf}

We need the following modification of the above construction. % with a symmetric measure. % of the above proposition to obtain a symmetric measure. 

\begin{cor}%[Grivaux's example revisited]
\label{cor:ex-Sophie-symm-set}
Let $E\subset \T$ be a Kronecker set with a continuous probability measure $\mu$ supported on it. Consider the symmetric probability measure $\nu$ on $E\cup \ol{E}$ defined by $\nu(A):=\frac12(\mu(A)+\mu(\ol{A}))$. Then the weak limit semigroup of the multiplicator $M_z$ on $L^2(\T,\nu)$ coincides with the set of contractions on $L^2(\T,\nu)$ commuting with $M_z$ that leave the closed real subspace  
\begin{equation}\label{eq:H-def}
\mathcal{H}:=\{f\in L^2(\T,\nu):\, f(\ol{z})=\ol{f(z)}\text{ for $\nu$-a.e.~}z\in\T\}
\end{equation}
invariant. 
\end{cor}
\begin{proof}
The inclusion ``$\subset$'' follows from the fact that $M_z$ leaves $\mathcal{H}$ invariant and that $\mathcal{H}$ is closed. Indeed, let $S : L^2(\T,\nu) \to L^2(\T,\nu)$ be the operator defined by $S(f)(z) = \overline{f(\overline{z})}$. Then $S$ is an isometry of $L^2(\T,\nu)$ and $\mathcal{H} = \ker(S-I)$ is closed. 

To show the converse inclusion,  let $V$ be a contraction on $L^2(\T,\nu)$ commuting with $M_z$ and leaving $\mathcal{H}$ invariant. Then, as in the proof of Proposition \ref{prop:ex-Sophie}, $V=M_f$ for some $f\in L^\infty(\T,\nu)$ with $\|f\|_{L^\infty(\T,\nu)} \leq 1$. Moreover, it is easy to see that $f\in \mathcal{H}$: Indeed, denote $A:=\{z:\, f(\ol{z})\neq\ol{f(z)}\}$ which is Borel. Then $A=\ol{A}$ and its characteristic function $1_A$ belongs to $\mathcal{H}$. By the assumption, we have 
$$
(f1_A)(\ol{z})=f(\ol{z})1_A(\ol{z})=\ol{f(z)1_A(z)},
$$
i.e., $\ol{f(z)}=f(\ol{z})$ for $\nu$-a.e.~$z\in A$, implying $\nu(A)=0$.

To show that $V$ belongs to the weak limit semigroup of $M_z$ on $L^2(\T,\nu)$, observe that by $f\in \mathcal{H}$ we have $\|f\|_{L^\infty(\T,\mu )}=\|f\|_{L^\infty(\T,\nu )}\leq 1$. So by Proposition \ref{prop:ex-Sophie}, there exists a subsequence $(n_k)$ of $\N$ with $\displaystyle \limk M_z^{n_k}=M_f$ weakly as operators on $L^2(\T,\mu)$. Take now $g,h\in L^2(\T,\nu)$ and denote by $C$ the operator on $L^2(\T,\nu)$ given by $Cu(z):=u(\ol{z})$. Then again by $f\in \mathcal{H}$, we can write $\la (M_z^{n_k}-V)g,h\ra_{L^2(\T,\nu)}$ as 
\begin{eqnarray*}
&\ &
%\la (M_z^{n_k}-V)g,h\ra_{L^2(\nu)} &=& 
\frac12\left(\int_\T (z^{n_k}-f(z)) g(z)\ol{h(z)}\,d\mu + \int_\T (\ol{z}^{n_k} - f(\ol{z})) g(\ol{z}) \ol{h(\ol{z})}\, d\mu\right)\\
&\ &\quad = \frac12\left(
\la (M_z^{n_k}-M_f)g,h\ra_{L^2(\T,\mu)} + \ol{\la (M_z^{n_k} - M_f) Ch,Cg\ra_{L^2(\T,\mu)}}
\right)
\end{eqnarray*}
which converges to zero. Thus $V$ belongs to the weak limit semigroup of $M_z$ on $L^2(\T,\nu)$.
\end{proof}

We will now lift the above example to the measure-preserving system induced by the corresponding stationary Gaussian process. Recall that for every symmetric probability measure $\nu$ on $\T$ there exists a 
%As in, e.g., \cite[proof of Prop.~3.1]{EisnerGrivaux11} we use the classical construction of a %real-valued 
stationary Gaussian process $(X_n)_{n\in\Z}$ on a probability space $(\Omega,\P)$ with the spectral measure $\nu$, see Peller \cite[Chapter 8]{Peller-book} for details and the notation we will use. This stationary Gaussian process can be realized as follows.
%By ....TODO... (see also \cite[proof of Prop.~3.1]{EisnerGrivaux11}), 
%There exists a real-valued stationary Gaussian process $(X_n)_{n\in\N}$ on a probability space $(\Omega,\P)$ with the spectral measure $\mu$ which can be realized as follows. 
The space $\Omega:=\R^\Z$ is endowed with the $\sigma$-algebra generated by cylinder sets of the form $\{(\omega_j)_{j \in \Z} :\, (\omega_{-m},\ldots,\omega_m)\in A\}$ for $m\in\N_{0}$ and Borel $A\subset \R^{2m+1}$, and denote by $T$ the left shift on $\Omega$. The $T$-invariant probability measure $\P$ on such a cylinder set above is the probability that the process at times $-m,\ldots,m$ belongs to $A$.
% $(X_{-m},\ldots,X_m)$ belongs to $A$.
 The stationary  Gaussian process is then realized as the sequence of projections defined by $X_n((\omega_j)_{j \in \Z}):=\omega_n$, $n\in\Z$, which satisfies $TX_n=X_{n+1}$ for the Koopman operator $T$ of the shift. Recall that the spectral measure $\nu$ is continuous if and only if the measure-preserving transformation $T$ on $(\Omega,\P)$ is ergodic (and, in fact, weakly mixing). 

We denote by $G_r$ the closure of the real linear combinations of $(X_n)$ and by $G:=G_r+iG_r$ its complexification. The map $\Phi$ defined by $\Phi(\sum_j c_jX_j):=\sum_j c_j z^j$ on finite (complex) linear combinations of $(X_n)$ extends to a unitary operator $\Phi:G\to L^2(\T,\nu)$ and the shift $T$ corresponds to the multiplicator $M_z$ on $L^2(\T,\nu)$ under $\Phi$. Note that $\Phi(G_r)=\mathcal{H}$, where $\mathcal{H}$ is defined in \eqref{eq:H-def}. Moreover, the orthogonal decomposition 
$
L^2(\Omega,\P)= \displaystyle \bigoplus_{n=0}^{\infty} \Gamma(G)_n
$
 holds, where the closed $T$-invariant subspaces $\Gamma(G)_n$ are defined via the so-called Vick transform and each $\Gamma(G)_n$ is unitarily isomorphic to the symmetric Hilbert tensor product space $G^n_{\tiny\textcircled{s}}$. Note that $\Gamma(G)_0$ is the subspace of constants and $\Gamma(G)_1=G$.
% of $G$ with itself $n$ times.
%
Every contraction/isometry $S$ on $G$ can be extended to a contraction/isometry $\mathcal{F}(S)$ on the Fock space $
%\mathcal{F}(G)=
\displaystyle \bigoplus_{n=0}^\infty G^n_{\tiny\textcircled{s}}$ by defining it as identity on constants and on each $G^n_{\tiny\textcircled{s}}$ for $n\geq 1$ via restriction 
%$S{\tiny\textcircled{s}}\cdots {\tiny\textcircled{s}}S$
 to symmetric tensors of the tensor products $S\otimes \ldots \otimes S$, inducing a contraction/isometry on $L^2(\Omega,\P)$. Note that, using this notation, %construction, the left shift on $G$ induces the left shift on $L^2(\Omega,\P)$, i.e., 
\begin{equation}\label{eq:Tn=F(Tn)}
T^n=\mathcal{F}(T^n|_G) \quad \forall n\in\Z,
\end{equation}
where we identify $L^2(\Omega,\P)$ with $\displaystyle \bigoplus_{n=0}^\infty G^n_{\tiny\textcircled{s}}$. For more details about the theory of Fock spaces we also refer to Bayart and Grivaux \cite[Subsection 3.4]{BayartGrivaux06}.

\smallskip

We now determine the weak limit semigroup of measure-preserving systems induced by stationary Gaussian processes with symmetric spectral measures as in Corollary \ref{cor:ex-Sophie-symm-set}. Note that such systems are called \emph{Gaussian-Kronecker systems}, see, e.g., Lema\'nczyk, Parreau, Thouvenot \cite{LemanczykParreauThovenot00} for the terminology. %, and go back\footnote{\color{blue} Girsanov worked with subsets of $\T$ without rational relations, and every Kronecker set is such.} to Girsanov \cite{Girsanov58}.  %, see Newton \cite{Newton66}.}

\begin{prop}\label{prop:Gauss-wlimsgr}
Let $T$ be the left shift on $(\Omega,\P)$ induced by a stationary Gaussian process with spectral measure $\nu$  as in Corollary \ref{cor:ex-Sophie-symm-set}. Then the weak limit semigroup of $T$ coincides with the set of all %positive 
contractions $V$ on $L^2(\Omega,\P)$ commuting with $T$ such that %$V\Eins=\Eins$, %both $V$ and $V^*$ leave $G_r$ invariant.  
$V$ leaves $G_r$ (and hence $G$) %and each of the subspaces $\Gamma(G)_n$
 invariant and $V=\mathcal{F}(V|_G)$, where we identify $L^2(\Omega,\P)$ with $\displaystyle \bigoplus_{n=0}^\infty G^n_{\tiny\textcircled{s}}$.
\end{prop}
%Note that the property $V=\mathcal{F}(V|_G)$ automatically implies $V\Eins=\Eins$.
\begin{proof}
For the inclusion ``$\subset$'' let $V=\displaystyle \limk T^{n_k}$ weakly.
Since real linear combinations of $(X_n)$ form a convex set, by Mazur's theorem its weak closure coincides with its norm closure, i.e., $G_r$ is weakly closed. This shows that $V$ leaves $G_r$ invariant. 
%and $f\in G_r$. Since $T$ leaves $G_r$ (and hence $G$) invariant, by the Hahn--Banach theorem $Vf\in G$. But since $\la Vf,g\ra = \limk\la T^{n_k}f,g\ra\in \R$ holds for every $g\in G_r$ (recall that $(X_n)$ is an orthogonal sequence), we have $Vf\in G_r$. 
For the property $V=\mathcal{F}(V|_G)$ observe first that $V\Eins=\Eins$. Moreover, for $n\geq 1$ and $f_1, \dotsc, f_n,  g_1, \dotsc, g_n\in G$,
$$
\la (T^{n_k}\otimes\cdots \otimes T^{n_k})(f_1\otimes\cdots\otimes f_n),g_1\otimes\cdots \otimes g_n\ra =\la T^{n_k}f_1,g_1\ra\cdots \la T^{n_k} f_n ,g_n\ra
$$
converges as $k\to\infty$ to 
$$
\la Vf_1,g_1\ra\cdots \la V f_n ,g_n\ra=\la (V\otimes\cdots \otimes V)(f_1\otimes\cdots\otimes f_n),g_1\otimes\cdots \otimes g_n\ra
$$
implying (by overgoing to symmetric tensors) that $\mathcal{F}(T^{n_k}|_G)$ converges weakly to $\mathcal{F}(V|_G)$ on $\displaystyle \bigoplus_{n=0}^\infty G^n_{\tiny\textcircled{s}}$. Thus $V=\mathcal{F}(V|_G)$ follows from \eqref{eq:Tn=F(Tn)}.
%The property $V=\mathcal{F}(V|_G)$ follows from the corresponding property for each $T^n$ and the definition of the tensor product $V$
%The rest is clear. //IS IT ???//

For the converse inclusion, let $V$ be a contraction on $L^2(\Omega,\P)$ commuting with $T$ such that  $V$ leaves $G_r$ 
%and each $\Gamma(G)_n$ 
invariant and $V=\mathcal{F}(V|_G)$. Then also $G$ is $V$-invariant.  %and $V\Eins=\Eins$. Moreover, $V^*\Eins=\Eins$, or, equivalently, $\int_\Omega Vf\,d\P=\int_\Omega f\,d\P$ holds since $L^2_0(\Omega,P)$ is $V$-invariant and $V\Eins=\Eins$. 
The contraction $\Phi V|_G\Phi^{-1}$ commutes with $M_z$ and leaves $\Phi(G_r)=\mathcal{H}$ invariant, so by Corollary \ref{cor:ex-Sophie-symm-set} there exists a subsequence $(n_k)$ of $\N$ such that 
\begin{equation}\label{eq:conv-Vfg}
\limk\la T^{n_k}f,g\ra =\limk \la \Phi^{-1}M^{n_k}_z\Phi f,g\ra = \la Vf,g\ra
\end{equation}
for every $f,g\in G$. 
%As in the proof of \cite[Chap.~8, Thm.~3.2]{Peller-book}, 
By using the identification of $L^2(\Omega,\P)$ and $\displaystyle \bigoplus_{n=0}^\infty G^n_{\tiny\textcircled{s}}$  
together with
 $T^{n_k}=\mathcal{F}(T^{n_k}|_G)$ and $V=\mathcal{F}(V|_G)$, 
\eqref{eq:conv-Vfg} extends to all $f,g\in L^2(\Omega,\P)$ as above. 
\end{proof}

\begin{remark}
Let $T$ be induced by an arbitrary stationary Gaussian process. 
As the above proof shows, every $V\in \Tlim$ is necessarily a contraction on $L^2(\Omega,\P)$ commuting with $T$ leaving $G_r$ (and hence $G$) invariant and satisfying $V=\mathcal{F}(V|_G)$, where we identify $L^2(\Omega,\P)$ with $\displaystyle \bigoplus_{n=0}^\infty G^n_{\tiny\textcircled{s}}$.
Thus, the weak limit semigroup of processes with spectral measure as in Corollary \ref{cor:ex-Sophie-symm-set} is maximal (among all stationary Gaussian processes).  
\end{remark}

As a corollary we obtain the following. 
\begin{thm}\label{thm:wlim-Gaussian-polyn}
Let $T$ be the left shift on $(\Omega,\P)$ induced by a stationary Gaussian process with spectral measure $\nu$  as in Corollary \ref{cor:ex-Sophie-symm-set}.
Then the weak limit semigroup of $T$ coincides with the set of all contractions $V$ of the form\footnote{The series here converges %with respect to the $\sot$ topology
strongly.} 
\begin{equation}\label{eq:V=Fock}
V=\mathcal{F}\left(\sum_{n=-\infty}^\infty a_nT^n|_G\right)
\end{equation}
for \emph{real} coefficients $a_n$. 
%on $L^2_0(\Omega,\P)$ contains all polynomials $P(T_0)$ of (the restriction of the Koopman operator) $T_0$ with non-negative coefficients having sum $\leq 1$. 
\end{thm}
%Note that, in particular, such transformations $T$ are $[0,1]$-rigid, but also include operators in the weak limit semigroup which coincide on $G$ with $-\lambda I$, $\lambda \in[0,1]$, and other negative operators of the above form.
\begin{proof}
For the inclusion ``$\subset$'', let $V$ be a contraction on $L^2(\Omega,\P)$ commuting with $T$ such that $V$ leaves $G_r$ (and hence $G$) %and each of the subspaces $\Gamma(G)_n$
 invariant and $V=\mathcal{F}(V|_G)$. By Proposition \ref{prop:Gauss-wlimsgr} it suffices to find $(a_n)\subset \R$ satisfying \eqref{eq:V=Fock}. Consider the operator $V|_G$ which we can write as an infinite matrix $(a_{n,m})_{n,m\in\Z}$ with respect to the orthonormal basis $(X_n)$ of $G$. We show that this matrix is constant on every diagonal, i.e., $a_{n,m}=a_{n+1,m+1}$ for every $n,m\in\Z$. Indeed, since $V$ commutes with $T$ we have
\begin{eqnarray*}
a_{n+1,m+1}&=&\la VX_{m+1},X_{n+1} \ra 
= \la VTX_m,X_{n+1} \ra\\
 &=&\la TVX_m,X_{n+1} \ra 
 =\la VX_m,T^{*}X_{n+1} \ra =\la VX_m,X_{n} \ra =a_{n,m}.
\end{eqnarray*}
Since the matrix corresponding to the left shift has ones on the lower diagonal and otherwise zero, the sequence $a_n:=a_{-n,-n}$ satisfies \eqref{eq:V=Fock}.

For the converse inclusion, let $V$ be a contraction of the form \eqref{eq:V=Fock}. By Proposition \ref{prop:Gauss-wlimsgr} we just need to check that $V$ commutes with $T$. This clearly holds on $G$
%Then $V|_G$ is a contraction commuting with $T|_G$
 implying by \eqref{eq:Tn=F(Tn)} and \eqref{eq:V=Fock}
$$
TV=\mathcal{F}(T|_G)\mathcal{F}(V|_G)=\mathcal{F}(TV|_G)=\mathcal{F}(VT|_G)=VT,
$$ 
where we again identify $L^2(\Omega,\P)$ with $\displaystyle \bigoplus_{n=0}^\infty G^n_{\tiny\textcircled{s}}$.
\end{proof}

\begin{remark}
%Thus we see that the weak limit semigroup of $T$ is \emph{much} larger than the set of all  contractions $V$ with $V_0=\sum_{n=-\infty}^\infty a_nT^n_0$ for non-negative coefficients having sum $\leq 1$. Indeed, on the reducing infinite-dimensional subspace $G$ also operators $-cI$ with $c\in [0,1]$ and, more generally, contractive linear combinations of the powers of $T$ with nonpositive coefficients are in the weak closure of $T|_G$. 
In particular, 
every such $T$ is $0$- and $1$-rigid but not\footnote{This is because of the tensor product definition of $\mathcal{F}$.} $\lambda$-rigid for every $\lambda\in (0,1)$, although $\Tlim$  contains operators coinciding with $\lambda I$ on $G$ for such $\lambda$. Moreover, it also includes operators in the weak limit semigroup which coincide on the $T$-reducing infinite-dimensional subspace $G$ with $-\lambda I$, $\lambda \in[0,1]$, and other negative linear combinations of the powers of $T$. % operators of the above form.
This might look contraintuitive since elements of the weak limit semigroup of $T$ are positive operators on  $L^2(\Omega,\P)$. Note however that 
%The following example might explain this phenomenon: The 
tensor product extensions of $-I$ of even degree equal $I$ which might explain this phenomenon. 
\end{remark}

We recall for completeness the following fundamental result of Halmos \cite[p.~77]{Halmos-book}.
\begin{prop}[Halmos' Conjugacy Lemma]\label{prop:Halmos-conjugacy}
   For every aperiodic\footnote{A transformation is called \emph{aperiodic} if the set of its periodic points has measure zero. It is easy to see that every ergodic transformation in $\TT$ is aperiodic.} %(e.g., ergodic) 
   transformation $T\in\TT$, its conjugacy class 
   $\{S^{-1}TS:\, S\in \TT\}$
   is dense in $\TT$.
\end{prop}

Since the transformation $T$ from Theorem \ref{thm:wlim-Gaussian-polyn} is ergodic (and even weakly mixing) by the continuity of the spectral measure and since %$\R^\Z$ is isomorphic to $[0,1]$
all standard non-atomic Borel probability spaces are isomorphic to $[0,1]$ (see, e.g., \cite[Remark 8.48]{EF-book} or \cite[Thm. 17.41]{Kechris-book}), Halmos' Conjugacy Lemma implies the following. 

\begin{cor}\label{cor:dense-real-coef}
For a dense set of measure-preserving transformations $T\in\TT$, 
%the weak limit semigroup of $T$ is much larger than the set of all operators $V$ with $V_0=\sum_{n=-\infty}^\infty a_nT^n_0$ for non-negative coefficients having sum $\leq 1$ in the following sense: 
there exists a closed $T$-reducing infinite-dimensional subspace $D$ of $L_0^2[0,1]$ such that all operators of the form\footnote{The series converges with respect to the operator norm topology by the summability condition on the sequence $(a_n)_{n \in \Z}$.}
$$
\sum_{n=-\infty}^\infty a_nT^n|_D
$$
with $a_n\in\R$ satisfying $\displaystyle \sum_{n=-\infty}^\infty |a_n|\leq 1$ belong to the restriction of the weak limit semigroup of $T$ to $D$.
\end{cor}

\section{Generic asymptotics for transformations}\label{sec:gen-transf}

The following example will be crucial for the rest of our paper, see Katok \cite[Thm.~I.2.1 and Sections~I.3.2 and I.4.4]{Katok03}, Ryzhikov \cite{Ryzhikov07b} and Goodson \cite{Goodson99}.
In fact, in Katok \cite[Thm.~I.2.1]{Katok03} a generic class of such transformations is presented.

\begin{prop}[Katok--Stepin transformation]\label{prop:ex-Katok}
    There exists a weakly mixing transformation $T\in\mathcal{T}$ such that the weak limit semigroup\footnote{Here we make use of Remark \ref{rem:prop-of-limits}(e).} $\Tlim$ contains operators of the form $\frac12(I+T^k)$ for every $k\in\N_0$. 
\end{prop}

%Note that it is also shown in Katok \cite[Thm.~I.2.1 and Sec.~I.4.4]{Katok03} also show that such transformations are generic.

We further need the following simple property which we state for completeness. 
\begin{lemma}
    \label{prop:polyn-Gdelta}
    Let $\mathcal{P}$ be a collection of polynomials. Then  transformations $T$ such that the limit semigroup $\Tlim$ contains $P(T)$ for every $P\in\mathcal{P}$ form a $G_\delta$ subset of $\mathcal{T}$.
\end{lemma}
\begin{proof} We can assume without loss of generality that each $P(T)$ is a contraction (otherwise the assertion holds trivially since the empty set is $G_\delta$). 
    Any collection of polynomials is clearly separable for the topology of uniform convergence on compact sets. Let $\mathcal{P}'$ be a dense countable subset of $\mathcal{P}$ and observe that $T$ is as desired if and only if it belongs to 
    \begin{equation}\label{eq:G-delta-transf}
\bigcap_{k,N, P} 
\left\{T\in \mathcal{T}:\, \exists \ n\geq N:\,   d(T^n,P(T))<\frac1k
\right\},
\end{equation}
where the intersection is taken over all $k,N\in\N$ and all polynomials $P\in \mathcal{P}'$ and  
$d$ denotes the metric inducing the weak topology on the space of contractions of $L^2[0,1]$ (see Subsection \ref{introsubsecoperators}). It is clear that each set $M_{k,N,P}$ under the intersection sign in \eqref{eq:G-delta-transf} is an open subset of $\TT$, concluding the proof.
\end{proof}

This gives the following result essentially due to  Ryzhikov \cite{Ryzhikov07b} and Katok \cite[Thm.~I.2.1 and Section~I.4.4]{Katok03}. Here, $\Pconst$ again denotes the orthogonal projection onto the constants. For the reader's convenience we present some details.

\begin{thm}\label{thm:generic-transf}
   Weakly mixing transformations with the property $\frac12(I+T^k)\in\Tlim$ for every $k\in \N_{0}$ form a dense $G_\delta$ set in $\mathcal{T}$. For every such $T$ the weak limit semigroup $\Tlim$ is convex and hence contains operators of the form $aP_{\text{const}} + a_0I+\ldots+a_kT^k$ for all $k\in \N$ and all non-negative $a,a_0,\ldots,a_k$ with sum equal to one.  
\end{thm}
\begin{proof}
The last assertion follows from Remark \ref{rem:prop-of-limits}(d), while  the $G_\delta$-property (in the first assertion) follows from Lemma \ref{prop:polyn-Gdelta} and the fact that weakly mixing transformations form a $G_\delta$ subset of $\TT$, see, e.g., Halmos \cite[p.~79]{Halmos-book}. 
    
For the density property in the first assertion, take a Katok--Stepin transformation from Proposition \ref{prop:ex-Katok}. By Halmos' Conjugacy Lemma (Proposition \ref{prop:Halmos-conjugacy})
%\cite[p.~77]{Halmos-book}, 
its conjugacy class is dense in $\TT$ and every element of it clearly satisfies the required property. The proof is complete.  
\end{proof}

%{\color{blue}}
Restricting to $L_0^2[0,1]$, by the weak closedness and $T^{-1}$-invariance of $\Tlim$ (see Proposition \ref{rem:prop-of-limits}(a)),  we obtain the following result.

\begin{cor}\label{cor:wlim-typ-transf} 
Transformations $T$ with the following property form a dense $G_\delta$ set in $\mathcal{T}$: 
%For a dense $G_\delta$ set of transformations $T\in\mathcal{T}$, 

The weak limit semigroup $\Tnulllim$ contains all operators of the form $\displaystyle \sum_{n=-\infty}^\infty a_nT_0^n$ with $(a_n)\in[0,1]$ satisfying $\displaystyle \sum_{n=-\infty}^\infty a_n\leq 1$. 
\end{cor}

%}
%\begin{proof}
%    By Theorem ..... such transformations form a residual set. The $G_\delta$-property follows analogously as in Lemma .... using $\TTnull$ instead of $\TT$.
%\end{proof}

\begin{remark}%{\color{red}//Shift to Introduction?//}
%It is known as a combination of two results\footnote{namely the Weak Closure Theorem for rank one transformations due to King \cite{King00}, see also Ryzhikov \cite{Ryzhikov07} for an alternative proof, and the genericity of rank one transformations due to Katok and Stepin \cite{KatokStepin67}.} that
As mentioned in the introduction,
typically the weak limit semigroup $\Tlim$ contains all (Koopman operators of) elements from $\mathcal{T}$ commuting with $T$. In particular, it contains all roots of $T$ in $\mathcal{T}$ and all (Koopman operators of) elements of every continuous\footnote{We call a flow $(T_t)$ of measure-preserving transformations on a probability space $(X,\mu)$ \emph{continuous} if it has strongly continuous Koopman semigroup, i.e., if for every $f\in L^2[0,1]$ the map $t\mapsto T_tf$, $\R_+\to L^2[0,1]$ is continuous. In this case,  the Koopman semigroup is strongly continuous in $L^p[0,1]$ for every $p\geq 1$.} 
flow $(T_t)$ embedding $T$. 
Note that such operators are actually strong$^*$ limit points of the powers 
since the strong$^*$ and the weak operator topologies coincide on $\mathcal{T}$. On the other hand, Theorem \ref{thm:generic-transf}  and Corollary \ref {cor:wlim-typ-transf} show that many more operators belong to the weak limit semigroup of a generic transformation which are not isometries and hence do not belong to the strong limit semigroup (see Remark \ref{rem:isom-strong-limit-sgr}). 
%One in particular recovers the result of Stepin that a typical transformation is $[0,1]$-rigid, see Corollary \ref{cor:lambda-rigid-typical}. 
\end{remark}

\section{Typical asymptotics for operators} \label{sec:genericasymptoticsforoperators}

Let $H$ be a separable infinite-dimensional Hilbert space. 
In this section, we study typical properties in the following Polish spaces: the space of unitary operators $\mathcal{U}$ on $H$ endowed with the $\sote$ topology, the space of isometries $\mathcal{I}$ on $H$ endowed with the $\sot$ topology, and the space of contractions $\mathcal{C}$ on $H$ endowed with the $\wot$ topology. 

%{\color{blue}We begin with the following result essentially due Grivaux, Matheron, Menet \cite{GrivauxMatheronMenet22}. The ressiduality property was partially done in }

The following refinement of the typicality of unitary transformations in $\mathcal{I}$ and $\mathcal{C}$ from \cite{E10} is essentially due to Grivaux, Matheron and Menet \cite{GrivauxMatheronMenet22}. We present here two proofs: a direct one inspired\footnote{We thank Ethan Akin for noticing this simplification of the argument in \cite{E10} in the ergodic theory setting of \cite{E25}.} by \cite{E10}, and another one using the points of continuity method based on \cite{GrivauxMatheronMenet22}.
%the details for the sake of completeness.
\begin{thm}[Dense $G_\delta$ property of unitary and isometric operators]\label{thm:unitary-G-delta-in-I-C}
Unitary operators form a dense $G_\delta$ set in the spaces $(\mathcal{I}, \emph{\sot})$ and $(\mathcal{C}, \emph{\wot})$. Moreover, isometric operators form a dense $G_\delta$ set in $(\mathcal{C}, \emph{\wot})$.
\end{thm}
\begin{proof}[First Proof of Theorem \ref{thm:unitary-G-delta-in-I-C} ]
 %One can also verify the $G_\delta$ property of unitaries in $\I$ and of isometries in $\mathcal{C}$ as follows, being a simplified version of the argument in \cite{E10} where the residuality of the corresponding sets was shown. 
 The density was shown in  \cite[Prop.~3.3 and 4.1]{EisnerSereny08}, so we just have to check the $G_\delta$ property.
    
   Let $(x_n)_{n \geq 1}$ be dense in $H$.
    %\setminus\{0\}$. 
    Since the range of every isometry is closed and $(Tx_n)_{n \geq 1}$ is dense in it, we have 
    $$
\U=\bigcap_{j,k \geq 1}\bigcup_{l \geq 1}\left\{T\in \I:\, \|x_j-Tx_l\|<\frac1k\right\}
=:\bigcap_{j,k \geq 1} \bigcup_{l \geq 1} M_{j,k,l}.
     $$
     It is clear that every set $M_{j,k,l}$ is $\sot$-open, showing the $G_\delta$ property of unitary operators in $\I$.

     To show the $G_\delta$ property of isometries in $\mathcal{C}$, observe first that
     $$
     \I=\bigcap_{j,k \geq 1} \left\{T\in\mathcal{C}:\, \|Tx_j\|>\left(1-\frac1k\right)\|x_j\|\right\}=:\bigcap_{j,k \geq 1} M_{j,k}.
     $$
    We show that for every $j,k\in\N$ we have $\I\subset \text{int}M_{j,k}$, the interior of $M_{j,k}$ for the $\wot$ topology. Assume that it is not so and there exists a sequence $(T_n)$ from the complement of $M_{j,k}$ %(which is clearly closed for the strong operator topology) 
     converging weakly to $T\in\I$. Since $T$ is an isometry, $(T_n)$ converges to $T$ strongly and hence $T\notin M_{j,k}$ by the $\sot$-closedness of the complement of $M_{j,k}$, a contradiction to the isometric property of $T$. Thus we have 
      $$
     \I=\bigcap_{j,k \geq1} \text{int}M_{j,k},
     $$
     proving the $G_\delta$ property of isometries in $\mathcal{C}$.

     Finally, to show the $G_\delta$ property of unitary operators in $\mathcal{C}$, observe that
     $$
     \U=\I\cap \I^*,
     $$
     where $\I^*:=\{T\in \mathcal{C}:\,T^*\in \I\}$.
     Since $\I$ is $G_\delta$ in $(\mathcal{C}, \wot)$ by the above, so is $\I^*$ since it is $\wot$-homeomorphic to $\I$ under the map $T\mapsto T^*$. So  $\U$ is a $G_\delta$ subset of $\mathcal{C}$ as the intersection of two such sets.
\end{proof}
\begin{proof}[Second Proof of Theorem \ref{thm:unitary-G-delta-in-I-C}]

    %The proof for $\mathcal{C}$ follows from and goes as follows.
    For topological spaces $X$ and $Y$, a map $f:X\to Y$  is called \emph{of Baire class $1$} (or \emph{Borel $1$}) if the preimage of every open subset of $Y$ is $F_\sigma$ in $X$. By Kechris \cite[Baire's Thm.~24.14]{Kechris-book}, the set of continuity points of a Baire class $1$ map $f:X\to Y$ forms a dense $G_\delta$ subset of $X$ whenever $X,Y$ are metrizable and $Y$ is separable. 

    By Grivaux, Matheron, Menet \cite[Prop.~2.11(2)]{GrivauxMatheronMenet22}, $\mathcal{U}$ coincides with the set of points of continuity of the identity map $\id:(\mathcal{C},\wot)\to(\mathcal{C},\sot^*)$. Moreover, this map is shown to be of Baire class $1$ in \cite[Lemma 2.9]{GrivauxMatheronMenet22}. Finally, $(\mathcal{C},\sot^*)$ is separable by \cite[beginning of Sect.~3]{GrivauxMatheronMenet21b}, see also the simplification\footnote{For our context, just replace $L^2(X,\mu)$ by $H$ and $\mathcal{T}$ by $\mathcal{C}$.} in \cite[Lemma 4.2]{E25}. Thus the assertion for the unitaries in $\mathcal{C}$ follows.

    The case of unitaries in $\mathcal{I}$ can be proved along the same lines for the identity map $\id:(\mathcal{I},\sot)\to(\mathcal{I},\sot^*)$. The same holds for the last assertion and the identity map $\id:(\mathcal{C},\wot)\to(\mathcal{C},\sot)$ using  \cite[Prop.~2.11(1)]{GrivauxMatheronMenet22}. 
\end{proof}

The following is an operator-theoretic analogue of Lemma \ref{prop:polyn-Gdelta} with analogous proof which we leave to the reader.
\begin{lemma}
    \label{prop:polyn-Gdelta-op}
    Let $\mathcal{P}$ be a collection of polynomials. Then  operators $T$ such that the weak limit semigroup $\Tlim$ contains $P(T)$ for every $P\in\mathcal{P}$ form a $G_\delta$ subset in $(\mathcal{U}, \emph{\sote})$ and $(\mathcal{I}, \emph{\sot})$. 
\end{lemma}
%\begin{proof}
%    For  $\mathcal{U}$ and $\mathcal{I}$ the proof goes analogously to the one of Lemma \ref{prop:polyn-Gdelta}. %, whereas the assertion for $\mathcal{C}$ follows then {\color{red} from \cite{E10} (or its stronger version in Theorem \ref{thm:unitary-G-delta-in-I-C}).}
%\end{proof}
%}}
We will now give an operator-theoretic analogue of Theorem \ref{thm:generic-transf} and Co\-rolla\-ry \ref{cor:wlim-typ-transf}.

\begin{thm}\label{thm:generic-op}%[Generic asymptotics in ergodic theory]
For a separable Hilbert space $H$, operators $T$ with the following property form a dense $G_\delta$ set in $(\mathcal{U}, \emph{\sote})$ and $(\mathcal{I},\emph{\sot})$, and 
a comeager set in  
$(\mathcal{C}, \emph{\wot})$:

The weak limit semigroup $\Tlim$ is convex and contains $P(T)$ for all polynomials $P$ with complex coefficients having sum of absolute values $\leq 1$. 
%$\mathcal{U}$, $\mathcal{I}$ and $\mathcal{C}$, respectively.}
\end{thm}
\begin{proof}
We have to show 
the assertion for $\mathcal{U}$ and $\mathcal{I}$. The assertion for $\mathcal{C}$ follows then from Theorem \ref{thm:unitary-G-delta-in-I-C}. 

By Lemma  \ref{prop:polyn-Gdelta-op}, operators having the required polynomials in the weak limit semigroup $\Tlim$ form a $G_\delta$ subset of $\I$ and $\U$, and for every  such operator $\Tlim$ is automatically convex by Remark \ref{rem:prop-of-limits}(d). 

To show the density, we assume without loss of generality that $H=L_0^2[0,1]$ and 
take the Koopman operator $T$ of a transformation satisfying the assertion of Theorem \ref{thm:generic-transf}. Since weak mixing implies ergodicity we have  $\sigma(T)=\T$, see, e.g., \cite[Prop.~3.32]{EF-book} or Proposition \ref{lemmespectreisometriesposbijaperiodic}. 
Thus $\sigma(T_0)=\T$, hence the conjugacy class of $T_0$ is dense in $\mathcal{U}$ by Proposition \ref{prop:conjugacy-lemma-op}, and every element $S$ of this conjugacy class satisfies 
\begin{align}
P(S) &\in \mathcal{L}_S, \; \forall \text{ polynomial } P \text{ with non-negative coefficients} \label{eq:posP-in-wlimsgr} \\
&\hspace{3.5em}\text{having sum } \le 1. \notag
\end{align}
Moreover, this property is $G_\delta$ in $\mathcal{U}$ by Lemma \ref{prop:polyn-Gdelta-op}. This shows that operators satisfying \eqref{eq:posP-in-wlimsgr} form a dense $G_\delta$ set in $\mathcal{U}$. Thus by Theorem \ref{thm:D-rigid-residual} there exists a $\overline{\D}$-rigid operator $U\in\mathcal{U}$ satisfying \eqref{eq:posP-in-wlimsgr}. Moreover, by typicality of unitary operators with full spectrum, see \cite{EisnerMatrai13}, we can assume without loss of generality that $\sigma(U)=\T$. In particular, $\lambda U^k\in \mathcal{L}_U$ holds for every $\lambda\in \overline{\D}$ and every $k\in\N_{0}$ by Proposition \ref{rem:prop-of-limits}(a). The convexity of $\mathcal{L}_U$ implies now that $P(U)\in \mathcal{L}_U$ for every complex polynomial $P$ with coefficients having sum of absolute values $\leq 1$. By Proposition \ref{prop:conjugacy-lemma-op}, the conjugacy class of $U$ is dense in $\mathcal{U}$, and every element of this class has the desired property, showing the density property in $\mathcal{U}$.
%
%observe that a typical Koopman transformation $T$ is ergodic (since weakly mixing) and hence satisfies $\sigma(T)=\T$, see, e.g.,  \cite[Prop.~3.32]{EF-book}. Thus take an ergodic transformation satisfying the desired property from Theorem \ref{thm:generic-transf}. Since both this property and the full spectrum property are preserved by unitary equivalence, 
%every Koopman operator $T$ from Theorem \ref{thm:generic-transf} is ergodic and hence satisfies $\sigma(T)=\T$ implying $\sigma(T_0)=\T$. Moreover, the weak limit semigroup $\Tnulllim$ contains all polynomials $P(T_0)$ of $T_0$ with non-negative coefficients having sum $\leq 1$. Since $L_0^2[0,1]$ is isomorphic to $H$, we obtain a unitary operator $U$ on $H$ with $\sigma (U)=\T$ satisfying the required property. Since this property is clearly conjugacy-invariant, Proposition \ref{prop:conjugacy-lemma-op} implies the density property in $\mathcal{U}$. 
Since  $\mathcal{U}$ is dense in $\mathcal{I}$ (see \cite{E10}), we also have the density in $\mathcal{I}$.
%
%The assertion for $\mathcal{C}$ again follows from the residuality of $\mathcal{U}$ in $\mathcal{C}$ from \cite{E10}. 
%Since $\mathcal{U}$ is residual in $\mathcal{C}$ by \cite{E10} and the weak and the strong$^*$ topologies coincide on $\mathcal{U}$, every set which is residual in $\mathcal{U}$ is also residual in $\mathcal{C}$. This shows the assertion for  $\mathcal{C}$.
%
%The last assertion follows from residuality of $\mathcal{U}$ in $\mathcal{I}$ and $\mathcal{C}$, the weak closedness of the weak limit semigroup of every $T$ and its invariance under $T^*=T^{-1}$ for unitary $T$ as observed in Remark \ref{rem:prop-of-limits}.
\end{proof}

\begin{cor} \label{cortypweaklimsemigroupinfinitesums}
 Unitary operators $T$ with the weak limit semigroup containing all operators of the form $\displaystyle \sum_{n=-\infty}^\infty a_nT^n$ with $(a_n)_{n \in \Z} \subseteq \C$ satisfying $\displaystyle\sum_{n=-\infty}^\infty |a_n|\leq 1$ form a dense $G_\delta$ set in $(\mathcal{U}, \emph{\sote})$ and $(\mathcal{I}, \emph{\sot})$, and 
a comeager set in  
$(\mathcal{C}, \emph{\wot})$.
\end{cor}
\begin{proof}
   The assertion for $\mathcal{U}$ follows directly from Theorem \ref{thm:generic-op} using
the weak closedness of the weak limit semigroup and its invariance under $T^*=T^{-1}$ for unitary $T$ as observed in Proposition \ref{rem:prop-of-limits}. The assertion for $\mathcal{I}$ and $\mathcal{C}$ follows then using Theorem \ref{thm:unitary-G-delta-in-I-C}.
\end{proof}

\section{Typical asymptotics for positive operators} \label{Section-genericposoperators}
This section is motivated by the fact that Koopman operators associated to invertible measure-preserving transformations are positive, in the sense that they map non-negative functions of $L^2(X,\mu)$ to non-negative functions of $L^2(X,\mu)$. We consider here positive contractions on the complex space $L^p[0,1]$ with $1 \leq p < \infty$ equipped with the Lebesgue measure $m$. 
Let us recall that for every $1 \leq p < \infty$, we denote by $\Cp$, $ \Ip$, and $\Up$ respectively the positive contractions, positive isometries, and positive surjective isometries of $L^p[0,1]$. We start by recalling some well-known facts about positive isometries of $L^p[0,1]$.

\subsection{Positive isometries of $\mathbf{L^p[0,1]}$ with $\mathbf{1 \leq p < \infty}$} \label{subsectionposisoLp}\par\noindent

A complete description of the positive isometries of $L^p[0,1]$, $1 \leq p < \infty$, is provided by Lamperti's Theorem \cite{Lamperti58}; see Royden's book \cite[Chapter 15, Section 7]{Royden-book} for the proof. These are namely the maps $T_\tau^{(p)} : L^p[0,1] \to L^p[0,1]$ given by $T_\tau^{(p)} f = h \cdot (f \circ \tau)$, where $\tau$ is a Borel measurable map from $[0,1]$ onto (almost all of) $[0,1]$ and $h \in L^p[0,1]$ is a non-negative function satisfying $\int_{\tau^{-1}(A)} h^p dm = \int_A dm$ for every Borel subset $A$ of $[0,1]$. The function $h$ is uniquely determined (up to a.e equivalence), and $\tau$ is uniquely determined (up to a.e equivalence) on the set where $h \ne 0$. Moreover, if the positive isometry is surjective, then $\tau$ is bijective (except on a set 
%(up to sets 
of measure $0$), $\tau^{-1}$ is measurable, $\tau$ and $\tau^{-1}$ take sets of measure $0$ into sets of measure $0$, and $h = (\frac{d m \circ \tau}{dm})^{1/p}$. 

From now on, we denote by $\omega_\tau$ the Radon-Nikodym derivative $\omega_\tau := \frac{d m\circ \tau}{dm}$, and such invertible transformations $\tau$ for which $m(A) = 0$ imply $m(\tau^{-1}(A)) = m(\tau(A)) = 0$ for every Borel subset $A$ are called \emph{non-singular transformations}. The set of non-singular transformations will be denoted by $\GG$. In particular, if $\tau \in \GG$ is measure-preserving, then $\wTau = \omega_{\tau^{-1}} = 1$. The map $\xi : \tau \mapsto T_\tau^{(p)}$ is a one-to-one map from $\GG$ onto $\Up$. We give $\GG$ the $p$-coarse topology defined as the $\sot$ topology on the corresponding invertible isometries of $L^p[0,1]$. 
%This topology is Polish and a compatible distance is given by $$\rho(\tau, \sigma) = \displaystyle \sum_{n \geq 1} 2^{-n} \lVert T_\tau^{(p)} \indic_{E_n} - T_\sigma^{(p)} \indic_{E_n} \rVert_p$$ 
%for every $\tau, \sigma \in \GG$, where $(E_n)_{n \geq 1}$ is an algebra generating the Borel subsets of $[0,1]$. 
It is known that the $p$-coarse topologies for $1 \leq p < \infty$ all coincide on $\GG$ (see \cite[Thm. 8]{ChoksiKakutani79}). Moreover, the set $\mathcal{T}$ is easily seen %\footnote{\color{blue}Indeed, $T\in \GG$ is Koopman if and only if it satisfies $T\Eins=\Eins$.} 
to be closed in $\GG$ in the weak topology. For more details about these topologies and the group of non-singular transformations, we refer to Choksi and Kakutani \cite{ChoksiKakutani79}, Ionescu \cite{Ionescu65} and Nadkarni \cite{Nadkarni-book}.

An important tool to prove density results in the class of non-singular transformations is the following analogue of Halmos' Conjugacy Lemma (Proposition \ref{prop:Halmos-conjugacy}).

\begin{proposition}[{\cite[Theorem 2]{ChoksiKakutani79}}] \label{lemma-conjugacynonsingular}
    For each aperiodic non-singular transformation $\tau \in \GG$, the conjugacy class $\{ \sigma^{-1} \tau \sigma : \sigma \in \GG \}$ is dense in $\GG$ in the coarse topology.
\end{proposition}
\enlargethispage{2\baselineskip}

The aim of this section is now to study the weak limit semigroup of a typical positive contraction/positive isometry/positive invertible isometry of $L^p[0,1]$ for $1 \leq p < \infty$. One can easily show that a typical positive contraction $T \in (\Cp, \sot)$ is strongly stable for every $1 \leq p < \infty$ (see, for instance, \cite[Lemma 5.9]{EisnerMatrai13}). We will thus work with the \wot\, topology on $\Cp$ when $1 < p < \infty$, as well as with \sot\, on $\mathcal{I}_p$ and $\U_p$.
\subsection{Positive operators on $L^p[0,1]$ with $p > 1$}\par\noindent

We start this subsection with the following important result from Grz\k{a}\'slewicz.

\begin{theorem}[{\cite[Thm. 2]{Gralewicz90}}] \label{thWOTdenseposisometriesbijgralew}
    For every $1 < p < \infty$, the set of positive invertible isometries of $L^p[0,1]$ is dense in $(\Cp, \emph{\wot})$.
\end{theorem}
We will deduce from Theorem \ref{thWOTdenseposisometriesbijgralew} that, for every $1 < p < \infty$, the sets $\I_p$ and $\U_p$ are dense $G_\delta$ in $(\Cp,\wot)$ and that $\U_p$ is dense $G_\delta$ in $(\I_p, \sot)$. 
In order to prove it, we will make use of points of continuity of the identity maps. 

Let $1 < p < \infty$. Given two topologies $\tau_1$ and $\tau_2$ on $\Cp$ among \wot, \sot\, and \, \sote, we denote by $\Contp(\tau_1, \tau_2)$ the set of points of continuity of the identity map from $(\Cp, \tau_1)$ onto $(\Cp, \tau_2)$. It is well known that for every $1 < p < \infty$, the sets $\Contp(\wot, \sot)$ and $\Contp(\wot, \sote)$ are \wot-dense $G_\delta$ in $\Cp$ (see, for instance, \cite[Cor. 2.10]{GrivauxMatheronMenet22}). Moreover, if $D \subset \Cp$ is $\tau_1$-dense and $\tau_2$-closed, then $\Contp(\tau_1, \tau_2) \subset D$ (\cite[Lemma 2.3]{GrivauxMatheronMenet22}). 
%Using Theorem \ref{thWOTdenseposisometriesbijgralew}, and using the fact that the set of positive isometries of $L^p[0,1]$ is \sot-closed in $\Cp$, we immediately obtain the following result.
We first prove the following lemma.

\begin{lemma} \label{lemmeadjointposetptscont}
    Let $1 < p < \infty$ and let $p'$ be the conjugate exponent of $p$. For every positive contraction $T \in \Cp$, the adjoint $T^*$ of $T$ on $L^{p'}[0,1]$ is a positive contraction. Moreover, $T \in \Contp(\emph{\wot}, \emph{\sote})$ if and only if $T \in \Contp(\emph{\wot}, \emph{\sot})$ and $T \in \textrm{Cont}_{p}(\emph{\wot}, \emph{\sott})$, where $\emph{\sott}$ is the $\emph{\sot}$ topology on the adjoints.  
\end{lemma}

\begin{proof}
The first part of the proof follows from the fact that $T^*$ is characterized by the following duality equation 
 $$
 \int_{[0,1]} Tf \cdot g \, dm = \int_{[0,1]} f \cdot T^*g \, dm 
 $$
for every $f \in L^p[0,1]$ and $g \in L^{p'}[0,1]$, together with the fact that a function $h \in L^{p'}[0,1]$ is non-negative if and only if $\int_{[0,1]} f h \, dm \geq 0$ for every non-negative function $f \in L^p[0,1]$. The second part of the proof is clear.
\end{proof}

It is not hard to see that every positive isometry of $L^p[0,1]$ with $1<p<\infty$ belongs to $\Contp(\wot,\sot)$. Indeed, this can be deduced from \cite[Prop. 3.32]{Brezis-book}, which states that in a uniformly convex Banach space $E$, every sequence $(x_n)_{n \geq 1}$ in $E$ converging weakly to a point $x \in E$ such that $\displaystyle \limsup_{n \to \infty} \lVert x_n \rVert \leq \lVert x \rVert$ converges strongly. Since $L^p[0,1]$ is uniformly convex for every $1 < p < \infty$, we obtain that $\I_p \subset \Contp(\wot,\sot)$ for every $1 < p < \infty$. This fact was already observed in a more general setting in \cite[Subsection 2.5]{GrivauxMatheronMenet22}. Moreover, for every $1 < p < \infty$, the topologies $\wot$ and $\sot$ coincide on $\I_p$ and the topologies $\wot, \sot$ and $\sote$ coincide on $\U_p$.

\begin{theorem} \label{thmtypicalWOTbijisometries}
    For every $1 < p < \infty$, we have $\Contp(\emph{\wot}, \emph{\sot}) = \I_p$ and $\Contp(\emph{\wot}, \emph{\sote}) = \U_p$. In particular, for every $1 < p < \infty$, the sets $\I_p$ and $\U_p$ are dense $G_\delta$ in $(\Cp, \emph{\wot})$. 
\end{theorem}

\begin{proof}
    By Theorem \ref{thWOTdenseposisometriesbijgralew}, the sets $\I_p$ and $\U_p$ are dense in $(\Cp, \wot)$. Moreover, the sets $\I_p$ and $\U_p$ are respectively closed in $(\Cp, \sot)$ and $(\Cp,\sote)$. Indeed, this is clear for $\I_p$, and for $\U_p$ this follows from the fact that if $T$ is a positive invertible isometry of $L^p[0,1]$ then $T^*$ is a positive invertible isometry of $L^{p'}[0,1]$. This implies in particular that $\Contp(\wot, \sot) \subset \I_p$ and $\Contp(\wot, \sote) \subset \U_p$. Now by the above remark, the set $\I_p$ is contained in $\Contp(\wot, \sot)$ for every $1 < p < \infty$, implying that $\I_p = \Contp(\wot, \sot)$. Finally, if $T \in \U_p$, then $T^*$ is a positive invertible isometry of $L^{p'}[0,1]$, implying that $T^*$ belongs to $\textrm{Cont}_{p'}(\wot, \sot)$ and thus that $T$ belongs to $\Contp(\wot, \sott)$. It follows that $\U_p = \Contp(\wot, \sote)$ by Lemma \ref{lemmeadjointposetptscont}.
\end{proof}

\begin{remark} \label{remproptypCpUp}
    In particular, a property of elements of $\Cp$ is typical in $(\Cp, \wot)$ if and only if it is typical in $(\U_p, \sot)$, by Theorem \ref{thmtypicalWOTbijisometries}.
\end{remark}

We also have an analogue of Theorem \ref{thmtypicalWOTbijisometries} in the class of positive isometries. 

\begin{theorem}
    For every $1 < p < \infty$, the set of points of continuity of the identity map from $(\I_p, \emph{\sot})$ onto $(\I_p, \emph{\sote})$ is equal to $\U_p$. In particular, for every $1 < p < \infty$, the set $\U_p$ is dense $G_\delta$ in $(\I_p, \emph{\sot})$.
\end{theorem}

\begin{proof}
    Every positive invertible isometry of $L^p[0,1]$ is a point of $(\sot, \sote)$-continuity by the above remark. For the converse inclusion, let $T \in \I_p$ be a point of $(\sot, \sote)$-continuity in $\I_p$. By Theorem \ref{thWOTdenseposisometriesbijgralew}, there exists a sequence $(T_n)_{n \geq 1}$ in $\U_p$ converging to $T$ for $\wot$. Since $T$ is an isometry the convergence also holds for $\sot$ and thus for $\sote$, since $T$ is a point of continuity. It implies that $T^*$ belongs to $\mathcal{I}_{p'}$ and thus that $T$ belongs to $\U_p$.
\end{proof}
\enlargethispage{2\baselineskip}

From Theorem \ref{thmtypicalWOTbijisometries}, we can provide a description of the spectrum of a \wot-typical positive contraction of $L^p[0,1]$. It is based on the following result, taken from the book \cite[Paragraph~3.6]{Nadkarni-book} and adapted to $L^p[0,1]$. It makes use of Rokhlin's lemma \cite[Lemma 4]{ChaconFriedman65}, stating that if $\tau \in \GG$ is an aperiodic non-singular transformation, then given $\varepsilon>0$ and an integer $n \geq 1$, there exists a measurable set $A$ such that $A, \tau A, \dotsc, \tau^{n-1} A $ are pairwise disjoint and $m([0,1] \setminus \displaystyle \bigcup_{i=0}^{n-1} \tau^i A) < \varepsilon$. 

\begin{lemma} \label{lemmespectreisometriesposbijaperiodic}
    Let $\tau \in \GG$ be aperiodic. Then the approximate spectrum (and hence the spectrum) of the positive invertible isometry $T_\tau^{(p)}$ of $L^p[0,1]$ is the whole unit circle.
\end{lemma}

\begin{proof}
    Let $\lambda \in \T$ and let $\varepsilon>0$. We will construct a function $f$ in $L^p[0,1]$ with $\lVert f \rVert_p = 1$ such that $\lVert T_\tau^{(p)} f - \lambda f \rVert_p < \varepsilon$. 

    Let $\delta < (\varepsilon/4)^{p}$ and let $n \geq 1$ such that $1/n < (\varepsilon/4)^{p}$. By Rokhlin's lemma, there exists a Borel subset $A$ such that $A, \tau A, \dotsc, \tau^{n-1} A$ are pairwise disjoint and $m([0,1] \setminus \displaystyle \bigcup_{k=0}^{n-1} \tau^k A) < \delta$. Let $C:= [0,1] \setminus \displaystyle \bigcup_{k=0}^{n-1} \tau^k A $.
    We set $f = a $ on $A$, with $a$ to be chosen later, and define $f$ inductively by setting
    $$
    f(x) = \lambda \left( \frac{d m\circ \tau^{-1}}{dm}(x) \right)^{1/p} f(\tau^{-1} x) \quad \textrm{for every $x \in \tau^k A$, $1 \leq k \leq n-1$},
    $$
    and $f = 1$ on $C$. It implies, in particular, that
    $$
    \int_{\tau^k A} \lvert f \rvert^p dm = \int_{\tau^{k-1} A} \lvert f \rvert^p dm, \, \forall 1 \leq k \leq n-1.
    $$
    Thus $\lVert f \rVert_p^p = n a^p m(A) + m(C)$. We choose $a = \left (\frac{1-m(C)}{n m(A)} \right)^{1/p}$ so that $\lVert f \rVert_p = 1$. Let us remark that 
    \begin{align} \label{eq1spectrumposisometries}
        T_\tau^{(p)} f (x) = \lambda f(x) \quad \textrm{for every  $x \in \displaystyle \bigcup_{k=0}^{n-2} \tau^k A$},
    \end{align}
    and that
    \begin{align} \label{eq2spectrumposisometries}
        \int_{\tau^k A} \lvert f \rvert^p dm = a^p m(A) \leq 1/n \quad \textrm{for every $ 0 \leq k \leq n-1 $}.
    \end{align}
    Thus
    \begin{align*}
        \lVert T_\tau^{(p)} f - \lambda f \rVert_p^p &= \int_{C \cup \tau^{n-1}(A)} \lvert T_\tau^{(p)}f - \lambda f \rvert^p dm \\
        &\leq \left( \left(\int_{C \cup \tau^{n-1} A} \lvert T_\tau^{(p)}f \rvert^p dm \right)^{1/p} +  \left(\int_{C \cup \tau^{n-1} A} \lvert f \rvert^p dm \right)^{1/p}   \right)^p \\
        &\leq \left( (m(C) + 1/n)^{1/p} + \left(\int_{C \cup \tau^{n-1} A} \lvert T_\tau^{(p)} f \rvert^p dm \right)^{1/p}  \right)^p.
    \end{align*}
    Moreover, $C \cup \tau^{n-1}A = [0,1] \setminus \displaystyle \bigcup_{k=0}^{n-2} \tau^k A$, so that
    \begin{align*}
        \int_{C \cup \tau^{n-1} A} \lvert T_\tau^{(p)} f\rvert^p dm &= \int_{[0,1]} \lvert T_\tau^{(p)} f\rvert^p dm - \displaystyle \sum_{j=0}^{n-2} \int_{\tau^j A} \frac{d m\circ \tau}{dm}(x) \lvert f( \tau x) \rvert^p dm(x) \\
        &= 1 - \displaystyle \sum_{j=0}^{n-2} \int_{\tau^{j+1} A} \lvert f \rvert^p dm \\
        &= m(C) + a^p m(A) \\
        &\leq m(C) + 1/n,
    \end{align*}
    therefore implying that
    $\lVert T_\tau^{(p)} f - \lambda f \rVert_p \leq 2 (\delta + 1/n)^{1/p} < \varepsilon$.
\end{proof}

We can deduce the following result.

\begin{proposition} \label{propspectrumtypicalposcontracWOT}
Let $1 < p < \infty$. The set of positive contractions $T \in \mathcal{C}_p$ such that $\sigma_{ap}(T) = \sigma(T) = \T$ is dense $G_\delta$ in $(\I_p, \emph{\sot})$ and $(\U_p, \emph{\sot})$, and comeager in $(\mathcal{C}_p, \emph{\wot})$.
\end{proposition}

\begin{proof}
The set of positive contractions of $L^p[0,1]$ such that $\sigma_{ap}(T) = \sigma(T) = \T$ can be written as
\begin{align*}
    \displaystyle \bigcap_{\lambda \in \Lambda} \bigcap_{K \geq 1} \{ T \in \Cp : \; \exists f \in L^p[0,1], \;\lVert Tf-\lambda f  \rVert_p < \frac{1}{K} \lVert f \rVert_p \},
\end{align*}
where $\Lambda$ is countable and dense in $\T$. In particular, this set is $G_\delta$ in $(\I_p, \sot)$ and in $(\U_p, \sot)$. Moreover, since the set of positive aperiodic invertible isometries of $L^p[0,1]$ is dense in $\U_p$ by \cite[Lemma 10]{ChaconFriedman65} \footnote{recall that the $p$-coarse topologies coincide on $\GG$} and since $\U_p$ is comeager in $(\mathcal{C}_p,\wot)$ by Theorem \ref{thmtypicalWOTbijisometries}, the conclusion of Proposition \ref{propspectrumtypicalposcontracWOT} follows by Remark \ref{remproptypCpUp}.
\end{proof}

The set of periodic non-singular transformations is also dense in the $p$-coarse topologies (\cite[Prop. 3]{Ionescu65} and \cite[Page 295]{ChaconFriedman65}). Thanks to this result, we obtain that a typical positive contraction of $L^p[0,1]$ is rigid.

\begin{proposition} \label{proptypposcontractionwotrigid}
Let $1 < p < \infty$.
The set of positive rigid contractions of $L^p[0,1]$ is dense $G_\delta$ in $(\I_p, \emph{\sot})$ and $(\U_p, \emph{\sot})$, and comeager in $(\mathcal{C}_p, \emph{\wot})$.
\end{proposition}

\begin{proof}
The set of positive rigid contractions of $\Cp$ is $G_\delta$ in $(\I_p, \sot)$ and $(\U_p, \sot)$ by the same argument as in the proof of Lemma \ref{prop:polyn-Gdelta-op}. This set is also dense in $\I_p$ and $\U_p$ since periodic non-singular transformations are dense in $\GG$. The conclusion of Proposition \ref{proptypposcontractionwotrigid} thus follows from Remark \ref{remproptypCpUp}.
\end{proof}
\enlargethispage{2\baselineskip}

Proposition \ref{proptypposcontractionwotrigid} was already observed in \cite{AgeevSilva01} for the class $\U_p$.
We now focus on $0$-rigidity. Let us recall that a contraction $T$ on a separable reflexive Banach space $X$ is $0$-rigid if and only if $T$ has no eigenvalues on the unit circle (\cite[Thm. II.4.1, p.~61]{E-book}). We thus investigate the point spectrum of a \wot-typical positive contraction of $L^p[0,1]$ and start with the following (well-known) lemma.
\begin{lemma} \label{lemmevaleurspropresnonsingulartransformations}
    Let $1 \leq p < \infty$ and let $\tau \in \GG$ be a non-singular transformation. The positive invertible isometry $T_\tau^{(p)}$ of $L^p[0,1]$ has an eigenvalue if and only if $1$ is an eigenvalue of $T_\tau^{(p)}$. If this is the case, then there exists a $\tau$-invariant probability measure $\nu$ on $[0,1]$ which is absolutely continuous with respect to the Lebesgue measure $m$.
\end{lemma}
\begin{proof}
The first part of Lemma \ref{lemmevaleurspropresnonsingulartransformations} follows from the fact that $\lvert T_\tau^{(p)} f \rvert = T_\tau^{(p)} \lvert f \rvert $ for every function $f \in L^p[0,1]$. For the second part, let $f \in L^p[0,1]$ satisfy $\lVert f \rVert_p = 1$ and $T_\tau^{(p)} f = f$. Then the function $g := \lvert f \rvert^p$ belongs to $L^1[0,1]$ and $\int_{[0,1]} g \, dm = 1$. Moreover, we have that $g = \omega_\tau \cdot g \circ \tau$. Let $\nu$ be the probability measure defined by $d\nu = g \, dm$. Then $\nu$ is $\tau$-invariant, since for every Borel subset $A$,
    \begin{align*}
        \nu(\tau(A)) &= \int_{\tau(A)} g(x) dm(x) \\
        &= \int_A g(\tau y) \, \omega_\tau(y) dm(y) \\
        &=\int_A g(y) dm(y) \\
        &= \nu(A).
    \end{align*}
    This concludes the proof of Lemma \ref{lemmevaleurspropresnonsingulartransformations}.
\end{proof}

From the work of Ionescu-Tulcea \cite[Thm. 3]{Ionescu65}, and Chacon and Friedman \cite[Thm. 3]{ChaconFriedman65}, it is well known that, in the $p$-coarse topology, a typical non-singular transformation $\tau \in \GG$ does not admit a $\sigma$-finite invariant measure absolutely continuous with respect to $m$. This is linked to Birkhoff's ergodic theorem. Precisely, if $\tau \in \GG$ has a non-trivial $\sigma$-finite invariant measure absolutely continuous with respect to $m$, then $\displaystyle \lim_{n \to \infty} \frac{1}{n} \displaystyle \sum_{k=0}^{n-1} T_{\tau^k}^{(1)} f(x)$ must exist and be finite a.e on a set of positive measure for every function $f \in L^1[0,1]$ (see \cite[Lemma 7]{ChaconFriedman65}). We can thus deduce the following result on the point spectrum of a \wot-typical positive contraction.

\begin{theorem} \label{theoremevaleursproprestypicalposcontrWOT}
Let $1 < p < \infty$. The set of positive contractions of $L^p[0,1]$ with no eigenvalues and hence $0$-rigid is dense $G_\delta$ in $(\I_p, \emph{\sot})$ and $(\U_p, \emph{\sot})$, and comeager in $(\Cp, \emph{\wot})$. 
\end{theorem}

\begin{proof}
By Lemma \ref{prop:polyn-Gdelta-op}, the set of positive isometries of $L^p[0,1]$ without eigenvalues is $G_\delta$ in $(\I_p, \sot)$ and $(\U_p, \sot)$.
It is also dense in $\I_p$ and $\U_p$ by Lemma \ref{lemmevaleurspropresnonsingulartransformations} and by \cite[Thm. 3]{Ionescu65} and \cite[Thm. 3]{ChaconFriedman65}, proving the first part of Theorem \ref{theoremevaleursproprestypicalposcontrWOT}. The second part follows again from Remark \ref{remproptypCpUp}.
\end{proof}

\begin{remark}
    We can give two other proofs of the fact that a typical positive contraction $T \in (\mathcal{C}_p, \wot)$ satisfies $\sigma(T) = \T$, based on Theorem \ref{theoremevaleursproprestypicalposcontrWOT}. The first proof is based on the fact that any positive contraction of $L^p[0,1]$, $1 < p < \infty$, such that $\sigma(T) \cap \T \ne \T$ and $\sigma_p(T^*) \cap \T = \emptyset$ is strongly stable. This follows from the theorem of Arendt–Batty–Lyubich–V\~u (\cite[Thm. 2.18, p.~51]{E-book}), together with the fact that the peripheral spectrum of a nonzero positive operator on $L^p[0,1]$ is cyclic (see \cite[Cor. 2.20, p.~52]{E-book} and \cite[Section V.4]{Schaefer-book} for more details). Since a typical positive contraction $T \in (\mathcal{C}_p,\wot)$ is unitary, it is not strongly stable and thus satisfies $\sigma(T) = \T$, since it has no eigenvalues (recall that the $\wot$ topology is self-adjoint). The second proof of this result is based on H\"olz \cite[Corollary 2.8]{Holz25} and the fact that a typical positive contraction is a positive invertible isometry with no eigenvalues and thus is a non-periodic Banach lattice homomorphism.
\end{remark}

We have seen in Theorem \ref{thm:D-rigid-residual} that $\overline{\D}$-rigidity is typical for a contraction on a separable Hilbert space. Thanks to Proposition \ref{proptypposcontractionwotrigid} and Theorem \ref{theoremevaleursproprestypicalposcontrWOT}, we will be able to deduce that $[0,1]$-rigidity is also typical in $(\mathcal{C}_p, \wot)$.
We first prove the following analog of Theorem \ref{thm:generic-op} for positive contractions, based on Proposition \ref{lemma-conjugacynonsingular}.

\begin{theorem} \label{typweaklimsemigpconvexposcontractionsWOT}
Let $1 < p < \infty$.
    The set of positive contractions $T \in \mathcal{C}_p$ such that $\frac{1}{2}( I + T^k)$ belongs to the weak limit semigroup of $T$ for every $k \geq 0$ is dense $G_\delta$ in $(\I_p, \emph{\sot})$ and $(\U_p, \emph{\sot})$, and comeager in $(\Cp, \emph{\wot})$. In particular, the set of positive contractions $T \in \mathcal{C}_p$ such that the weak limit semigroup $\mathcal{L}_T$ is convex and contains $P(T)$ for all polynomials P with non-negative coefficients having sum less than or equal to $1$ is dense $G_\delta$ in $(\I_p, \emph{\sot})$ and $(\U_p, \emph{\sot})$, and comeager in $(\Cp, \emph{\wot})$. 
\end{theorem}

\begin{proof}
    The set $\mathcal{A}_p$ of positive contractions $T \in \mathcal{C}_p$ such that $\frac{1}{2}( I + T^k)$ belongs to the weak limit semigroup of $T$ for every $k \geq 0$ is $G_\delta$ in $(\I_p, \sot)$ and $(\U_p, \sot)$ by the same argument of the proof of Lemma \ref{prop:polyn-Gdelta-op}. For the density of $\A_p$, let $\tau$ be a weakly mixing measure-preserving transformation on $[0,1]$ whose weak limit semigroup contains all the operators $\frac{1}{2}(I + 
    (T_\tau^{(2)})^{k})$ with $k \geq 0$ (given by Proposition \ref{prop:ex-Katok}). This implies that for every $k \geq 0$, there exist a subsequence $(n_l)_{l \geq 1}$ in $\N$ such that 
    \begin{align} \label{eqweaklimsemigpL^2}
        \displaystyle \int_{[0,1]} f \circ \tau^{n_l} \cdot g \, dm \underset{l \to \infty}{\longrightarrow} \frac{1}{2} \left ( \int_{[0,1]} f \cdot g \, dm + \int_{[0,1]} f \circ \tau^k \cdot g \, dm \right)
    \end{align}
    for every $f, g \in L^2[0,1]$. Let $p'$ be the conjugate exponent of $p$. Since $(\ref{eqweaklimsemigpL^2})$ holds for every simple functions $f, g$ on $[0,1]$, it also holds for every $f \in L^p[0,1]$ and $g \in L^{p'}[0,1]$, since simple functions are dense in $L^p[0,1]$ and $L^{p'}[0,1]$. Hence the weak limit semigroup of $T_\tau^{(p)}$ contains all the operators $\frac{1}{2}(I + 
    (T_\tau^{(p)})^{k})$ with $k \geq 0$. Since the conjugacy class $\{ T_\sigma^{(p)} T_\tau^{(p)} (T_\sigma^{(p)})^{-1} : \sigma \in \GG \}$ is dense in $(\U_p, \sot)$ by Proposition \ref{lemma-conjugacynonsingular} and since this conjugacy class is contained in $\mathcal{A}_p$, we deduce that the set $\A_p$ is dense $G_\delta$ in $(\I_p, \sot)$ and $(\U_p, \sot)$, and comeager in $(\Cp, \wot)$ by Remark \ref{remproptypCpUp}. 

    For the second part of the proof, the set of positive contractions $T \in \Cp$ such that the weak limit semigroup $\mathcal{L}_T$ is convex and contains $P(T)$ for all polynomials $P$ with non-negative coefficients having sum less than or equal to $1$ is $G_\delta$ in $(\I_p, \sot)$ and $(\U_p, \sot)$ by the same proof as in Lemma \ref{prop:polyn-Gdelta-op} and by Proposition \ref{rem:prop-of-limits}(d). Moreover, this set is dense in $(\U_p, \sot)$, since the set of positive $0$-rigid contractions $T \in \Cp$ such that the weak limit semigroup contains all the operators $\frac{1}{2}(I + T^k)$ with $k \geq 0$ is comeager in $(\U_p, \sot)$ by the first part of the proof and by Theorem \ref{theoremevaleursproprestypicalposcontrWOT}. This concludes the proof of Theorem \ref{typweaklimsemigpconvexposcontractionsWOT}.
\end{proof}

\begin{remark}
    In particular, Theorem \ref{typweaklimsemigpconvexposcontractionsWOT} and an analog of Lemma \ref{prop:polyn-Gdelta-op} imply that the set of $[0,1]$-rigid positive contractions in dense $G_\delta$ in $(\Up, \sot)$ and $(\Ip, \sot)$, and comeager in $(\Cp, \wot)$. 
Note that, since there is no positive isometric isomorphism between
$L^2[0,1]$ and $L^2_0[0,1]$, one cannot proceed as in the proof of Theorem~\ref{thm:D-rigid-residual} to prove that $[0,1]$-rigidity is typical in $(\Cp, \wot)$, $(\Ip, \sot)$ and $(\Up, \sot)$.
\end{remark}

As a consequence of Theorem \ref{typweaklimsemigpconvexposcontractionsWOT}, we also obtain the following analogue of Corollary \ref{cortypweaklimsemigroupinfinitesums}.

\begin{cor}
Let $1 < p < \infty$. The set of positive invertible isometries $T$ of $L^p[0,1]$ such that the weak limit semigroup contains $\displaystyle \sum_{n=-\infty}^{\infty} a_n T^n$
with $(a_n)_{n \geq 0} \subseteq [0,1]$ satisfying $\displaystyle \sum_{n=-\infty}^\infty a_n\leq 1$ is dense $G_\delta$ in $(\I_p, \emph{\sot})$ and $(\U_p, \emph{\sot})$, and comeager in $(\Cp, \emph{\wot})$.
\end{cor}

\begin{proof}
   The proof goes exactly as the proof of Corollary \ref{cortypweaklimsemigroupinfinitesums}.
\end{proof}

We end this subsection by remarking that, contrary to a generic measure-preserving transformation on $[0,1]$ where the weak limit semigroup of the associated Koopman operator contains the orthogonal projection onto the constants, the weak limit semigroup of a typical positive contraction $T \in (\mathcal{C}_p, \wot)$ does not contain any rank-one positive operator.

\begin{proposition} \label{projectionweaklimsemigrouptypicalWOT}
Let $1 < p < \infty$. The weak limit semigroup of a typical positive contraction $T \in (\I_p, \emph{\sot})$, $T \in (\U_p, \emph{\sot})$ and  $T \in (\mathcal{C}_p, \emph{\wot})$ does not contain any rank-one positive operator.
\end{proposition}

\begin{proof}
    Suppose that there exists a rank-one positive operator $P$ on $L^p[0,1]$ in the weak limit semigroup of a positive contraction $T \in \mathcal{C}_p$. Let us set $P(f) = \langle \varphi, f \rangle \, h$ for every $f \in L^p[0,1]$, where $\varphi \in L^{p'}[0,1]$ and $h \in L^p[0,1]$ are nonzero, and $\langle \varphi, \psi \rangle := \displaystyle \int_{[0,1]} \varphi \, \psi \, dm $ for $\psi \in L^p[0,1]$. Let $f \in L^p[0,1]$ be such that $\langle \varphi, f \rangle \ne 0$. Then, since $T$ and $P$ must commute, we obtain that $Th = \frac{\langle \varphi, T f \rangle}{\langle \varphi, f \rangle} h$, implying that $T$ has an eigenvalue. The proof of Proposition \ref{projectionweaklimsemigrouptypicalWOT} follows from Theorem \ref{theoremevaleursproprestypicalposcontrWOT}. 
\end{proof}

\subsection{Positive operators on $L^1[0,1]$.}\par\noindent

As mentioned at the end of Subsection \ref{subsectionposisoLp}, an \sot-typical positive contraction of $L^1[0,1]$
is strongly stable, implying that its weak limit semigroup is
essentially trivial.
The aim of this subsection is to study the weak limit semigroup of an
\sot-typical positive isometry of $L^1[0,1]$.
An important approximation theorem for positive isometries of
$L^1[0,1]$ is the following, proved by Iwanik.

\begin{theorem}[{\cite[Thm. 2]{Iwanik80}}] \label{ThiwanikisoposbijL1}
    The set of positive invertible isometries of $L^1[0,1]$ is dense in $(\I_1, \emph{\sot})$.
\end{theorem}

As a consequence of Theorem \ref{ThiwanikisoposbijL1}, we obtain that a typical positive isometry of $L^1[0,1]$ is invertible.

\begin{prop} \label{proptypicalposisoisbijSOT}
    The set $\U_1$ is dense $G_\delta$ in $(\I_1, \emph{\sot})$.
\end{prop}

\begin{proof}
    This easily follows from the fact that the set $\U_1$ of positive invertible isometries of $L^1[0,1]$ is \sot-$G_\delta$ given by
    $$
    \displaystyle \bigcap_{j,k \geq 1} \{ T \in \I_1 : \; \exists f \in L^1[0,1], \;\lVert Tf - f_j \rVert < 1/k \}, 
    $$
    where $(f_j)_{j \geq 1}$ is dense in $L^1[0,1]$. (Recall that isometries have closed range.) Moreover, this set is \sot-dense in $\I_1$ by Theorem \ref{ThiwanikisoposbijL1}. This concludes the proof of Proposition \ref{proptypicalposisoisbijSOT}.
\end{proof}

The same proof as Proposition \ref{propspectrumtypicalposcontracWOT} leads to the following result.

\begin{prop} \label{spectreisometrytypiqueposL1}
    The set of positive isometries $T \in \I_1$ satisfying $\sigma_{ap}(T) = \sigma(T) = \T$ is dense $G_\delta$ in $(\I_1, \emph{\sot})$ and $(\U_1, \emph{\sot})$.
\end{prop}

Using Lemma \ref{lemmevaleurspropresnonsingulartransformations} and again the results from Chacon, Friedman and Ionescu-Tulcea, we deduce the point spectrum of a typical positive isometry of $L^1[0,1]$.

\begin{prop} \label{propvaleurspropresisometriestypiquesposL^1}
   The set of positive isometries of $L^1[0,1]$ with no eigenvalues is comeager in $(\I_1, \emph{\sot})$ and $(\U_1, \emph{\sot})$.
\end{prop}

\begin{proof}
    By Lemma \ref{lemmevaleurspropresnonsingulartransformations} and \cite[Proof of Thm.~3]{Ionescu65}, the set of positive invertible isometries of $L^1[0,1]$ with a non-empty point spectrum is contained in $$\mathcal{N} := \displaystyle \bigcup_{k \geq 1} \bigcap_{k \leq n \leq m} \{ T \in \U_1 : \; \rho(T^{(n,m)}, 0) \leq 1/5 \},$$
    where $T^{(n,m)} (x) := \displaystyle \sup_{n \leq j \leq m} \frac{T^j \indic(x)}{j}$ and $\rho(f,g) :=\displaystyle \int_{[0,1]} \frac{\lvert f-g \rvert}{1 + \lvert f - g \rvert} dm$. Since, for every $n \leq m$, the map $T \mapsto T^{(n,m)}$ is continuous from $\U_1$ into the space of real measurable maps on $[0,1]$ equipped with the distance $\rho$, the set $\mathcal{N}$ is $F_\sigma$ in $(\U_1, \sot)$. Moreover, this set is also nowhere dense in $\U_1$ by \cite[Thm. 2]{Ionescu65}. Since $\U_1$ is comeager in $\I_1$ by Proposition \ref{proptypicalposisoisbijSOT}, this concludes the proof of Proposition \ref{propvaleurspropresisometriestypiquesposL^1}. 
\end{proof}

Let us remark that if $T$ is a positive isometry of $L^1[0,1]$, then $\int_{[0,1]} Tf dm = \int_{[0,1]} f dm$ for every $f \in L^1[0,1]$. Indeed, this relation holds for every non-negative function $f \in L^1[0,1]$ since $T$ is positive, and it extends to all $f \in L^1[0,1]$. In particular, it implies that $\lambda I$ cannot be a weak limit of $T^{n_k}$ with $(n_k) \subseteq \N$ when $\lambda \in [0,1)$. However, rigidity remains typical in $\I_1$ and $\U_1$ as the following result shows.

\begin{prop} \label{proprigidL^1sot}
    The set of rigid positive isometries of $L^1[0,1]$ is comeager in $(\I_1, \emph{\sot})$ and $(\U_1, \emph{\sot})$.
\end{prop}

\begin{proof}
Using that $\sot$ is metrizable on $\I_1$, we can show as in the proof of Lemma \ref{prop:polyn-Gdelta} that the set of positive isometries $T$ of $L^1[0,1]$ such that $I$ is the strong limit of $T^{n_k}$ for some subsequence $(n_k)_{k \geq 1}$ in $\N$ is $G_\delta$ in $(\I_1, \sot)$. Since non-singular periodic transformations are dense in $\GG$, this set is also $\sot$-dense in $\U_1$, implying the conclusion of Proposition \ref{proprigidL^1sot}.
\end{proof}

We can also derive from Theorem \ref{thm:generic-transf} the following density\footnote{The lack of metrizability of the $\wot$ topology on $\I_1$ prevents us from strengthening the density result into a comeager one.} result.

\begin{proposition} \label{propdensityL1seriesweaklimsemigroup}
    The set of positive invertible isometries $T \in \I_1$ such that the weak limit semigroup contains $\displaystyle \sum_{n = 0}^{N} a_n T^n$ with $N \geq 0$ and $(a_n)_{0 \leq n \leq N} \subseteq [0,1]$ satisfying $\displaystyle \sum_{n = 0}^{N} a_n = 1$ is dense in $(\I_1, \emph{\sot})$ and $(\U_1, \emph{\sot})$.
\end{proposition}

\begin{proof}
    Let $\tau$ be a weakly mixing measure-preserving transformation on $[0,1]$ whose weak limit semigroup contains all the operators $\displaystyle \sum_{n = 0}^{N} a_n (T_\tau^{(2)})^n$ with $N \geq 0$ and $(a_n)_{0 \leq n \leq N} \subseteq [0,1]$ satisfying $\displaystyle \sum_{n = 0}^{N} a_n = 1$ (given by Theorem \ref{thm:generic-transf}). As in the proof of Theorem \ref{typweaklimsemigpconvexposcontractionsWOT}, using the density of simple functions in $L^1[0,1]$ and in $L^\infty[0,1]$, the weak limit semigroup of $T_\tau^{(1)}$ contains all the operators $\displaystyle \sum_{n = 0}^{N} a_n (T_\tau^{(1)})^n$ with $(a_n)_{0 \leq n \leq N} \subseteq [0,1]$ satisfying $\displaystyle \sum_{n = 0}^{N} a_n = 1$. The conjugacy class of $T_\tau^{(1)}$ in $\U_1$ is dense in $(\U_1, \sot)$ by Proposition \ref{lemma-conjugacynonsingular} and we thus obtain the conclusion of Proposition \ref{propdensityL1seriesweaklimsemigroup}.
\end{proof}

We end this section with the following remark regarding the strong limit semigroup of a typical positive isometry of $L^1[0,1]$ analogous to Remark \ref{rem:isom-strong-limit-sgr}.

\begin{remark}
    The strong limit semigroup of a typical positive isometry $T \in (\I_1, \sot)$ does not contain any of the operators $\frac{1}{2}( I + T^k)$ with $k \geq 1$. Indeed, if such an operator $S$ belongs to the strong limit semigroup of $T \in \I_1$, then $S$ must be an isometry and thus $f \cdot T^k f \geq 0$ for every real function $f \in L^1[0,1]$, by the equality case of Minkowski's inequality in $L^1[0,1]$. Now let us consider $T = T_\tau^{(1)}$ in $\U_1$ with $\tau \in \GG$ aperiodic. Note that such $T$ are typical by Proposition \ref{spectreisometrytypiqueposL1}. By Rokhlin's lemma, there exists a Borel subset $A$ with $m(A) > 0$ such that $A, \tau(A), \dotsc, \tau^k(A)$ are pairwise disjoint and $m([0,1] \setminus \displaystyle \bigcup_{i=0}^{k} \tau^i (A)) < 1/2$. Let $f$ be the function defined by $f = 1$ on $A$, $f= -1$ on $\tau^k (A)$ and $f= 0$ elsewhere. Then $f(x) T^k f(x) < 0$ on $A$, implying that $\frac{1}{2}(I + T^k)$ does not belong to the strong limit semigroup of $T$. 
\end{remark}

\section{Additional results and open problems}

We provide additional results and open questions related to our study.

\subsection{Measure-preserving transformations}\par\noindent

%In the context of transformations, Theorem 6.3 shows that the weak limit semigroup of a generic transformation $T \in \mathcal{T}$ contains all the convex combinations of $\Pconst$ and the non-negative powers of $T$. One could naturally wonder 
The general question of describing all operators belonging to the weak limit semigroup of a generic transformation $T\in\TT$ remains open.
%what other functions of a generic transformation belongs to its weak limit semigroup as for example roots in the weak closure of $\mathcal{T}$. 
Note that the weak closure of $\mathcal{T}$ is precisely the set of all bi-stochastic operators on $L^2[0,1]$ by a result of Vershik, see \cite[Lemma 2.1]{E25}, so elements of the weak limit semigroup are necessarily bi-stochastic (and commute with $T$). Here, an operator $P$ on $L^2[0,1]$ is called \emph{bi-stochastic} if it is positive (in the above sense) and satisfies $P\Eins=P^*\Eins=\Eins$. 
Note further that isometric bi-stochastic operators are precisely (not necessarily invertible) Koopman operators.\footnote{Indeed, if $S$ is such an operator then we can write $Sf = h \cdot f \circ \tau$ with $\tau$ a Borel measurable map from $[0,1]$ onto $[0,1]$ and $h \in L^2[0,1]$ a non-negative function satisfying $\int_{\tau^{-1}(A)} h^2 dm = \int_A dm$ for every Borel subset $A$ of $[0,1]$. Since $S \indic = \indic$ we obtain $h = 1$ and thus $\tau$ is measure-preserving.}

King's Weak Closure Theorem \cite{King86} implies that for a generic transformation $T\in\TT$, the weak closure of the powers of $T$ contains all (not necessarily invertible) Koopman operators commuting with $T$, i.e., all bi-stochastic isometries commuting with $T$. Thus the following question of generic maximality of the weak limit semigroup is natural.

\begin{problem}\label{prob:generic-trans-comm}
    Is it true that for a generic transformation $T \in \mathcal{T}$, all bi-stochastic operators commuting with $T$ belong to $\mathcal{L}_T$ (i.e.,  $\mathcal{L}_T$ is the set of all such operators)?
\end{problem}
Note that by (generic) convexity of $\Tlim$ and King's Weak Closure Theorem as mentioned in the introduction, for example operators of the form $\alpha P_\const + \beta S$ for $\alpha,\beta\in[0,1]$ with $\alpha+\beta = 1$ and $S$ being a (not necessarily invertible) Koopman operator commuting with $T$ belong to $\Tlim$ for a generic $T\in\TT$.

We will show now that $T$ satisfying %the last 
the property from Problem \ref{prob:generic-trans-comm} form a $G_\delta$ set in $\TT$, leaving its density as an open problem.
For this we will need the following easy fact being a consequence of the Banach-Alaoglu theorem, see, e.g., \cite[Page 275, Proposition 5.5]{Conway-book}.

\begin{lemma}\label{lem:weakly-conv-subseq}
Every sequence of contractions on a separable Hilbert space $H$ has a weakly converging subsequence. 
\end{lemma}

\begin{proof}[Proof of $G_\delta$ property for Problem \ref{prob:generic-trans-comm}]
Denote the set of all bi-stochastic operators on $L^2[0,1]$ by $\mathcal{BS}$. Observe that $T\in\mathcal{T}$ is as desired if and only if it belongs to 
$$
\bigcap_{k,N\in\N} %\bigcup_{n\geq N} 
\left\{T\in \mathcal{T}:\, \forall   S\in\mathcal{\mathcal{BS}} \text{ commuting with $T$},\ \exists \ n\geq N:\,   %\sum_{j=1}^\infty\frac{\|T^nf_j-Sf_j\|}{2^{j}\|f_j\|}
d(T^n,S)<\frac1k\right\},
$$
where $d$ denotes the metric inducing the weak topology on $\mathcal{C}$. 
We show that each set $M_{k,N}$ under the intersection sign is open in $\TT$. Let $(T_m)\subset \mathcal{T}\setminus M_{k,N}$ with $\displaystyle \lim_{m\to\infty}T_m=T\in\mathcal{T}$ strongly$^*$ and denote by $S_m$ the corresponding bi-stochastic operators commuting with $T_m$ with $d(T_m^n,S_m)\geq \frac1k$
%$$
%\sum_{j=1}^\infty\frac{\|T_m^nf_j-S_mf_j\|}{2^{j}\|f_j\|}\geq \frac1k.
%$$  
for all $n\geq N$. By Lemma \ref{lem:weakly-conv-subseq} we can assume without loss of generality that $\displaystyle \lim_{m\to\infty} S_m=S$ weakly for some contraction $S$ (being clearly bi-stochastic, too) which commutes with $T$ (indeed, it is easy to see that $T_m S_m$ converges weakly to $TS$ and $S_m T_m$ converges weakly to $S T$).
Moreover,  we clearly have $d(T^n,S)\geq \frac1k$ for all $n\geq N$, showing $T\notin M_{k,N}$. This concludes the proof of the $G_\delta$ property for Problem \ref{prob:generic-trans-comm}. 
\end{proof}

%Note that by \cite[Lemma 2.1]{E25}, the weak closure of $\mathcal{T}$ coincides with the set of bi-stochastic operators on $L^2[0,1]$.

\subsection{Hilbert space operators}\par\noindent

For typical Hilbert space contractions, we have seen that the weak limit semigroup contains many operators such as infinite series. It is interesting and natural to ask which other functions of $T$ are there, such as for example the roots.
%Moreover, we have seen in Section \ref{sec:ex-op} examples of operators whose commutant is contained in the weak closure of their powers. Motivated by King's Rank-One Theorem, the following problem is open for contractions.
%
%\begin{problem}
%    Is it true that the commutant of a rigid operator of $\U$ is either trivial (i.e only contains the positive powers and the identity) or uncountable?
%\end{problem}
%Also motivated by King's Theorem, the following problem in the operator setting is open.

\begin{problem}
    Is it true that for a typical unitary operator $T$, all unitary operators $S$ satisfying $S^2 = T$ belong to $\mathcal{L}_T$?
\end{problem}

More generally, one can ask whether the maximality of the weak limit semigroup we saw in Example \ref{prop:ex-Sophie} is typical.
\begin{problem}
    Is it true that for a typical unitary operator $T$, all contractions commuting with $T$ belong to $\mathcal{L}_T$?
\end{problem}

As for Problem \ref{prob:generic-trans-comm}, the $G_\delta$ property easily holds and it is the density which is open.

As mentioned in the introduction, by King's Weak Closure Theorem the answer to both problems in the context of transformations is positive, where both the roots and the commuting contractions are assumed to be Koopman operators, too.

\subsection{Positive operators on $\ell_p$}\par\noindent

The second author previously studied typical properties of positive contractions on $\ell_p$-spaces with the aim of studying the Invariant Subspace Problem for such operators. More precisely, let $(e_n)_{n \geq 0}$ be the canonical basis of $\ell_p$ with $1 \leq p < \infty$. A vector $x = \displaystyle \sum_{n \geq 0} x_n e_n$ is said to be positive (written $x \geq 0$) if $x_n \geq 0$ for every $n \geq 0$, and an operator $T$ on $\ell_p$ is said to be positive if $T x \geq 0$ for every $x \geq 0$. He proved in \cite{gillet2} that for every $1 < p \leq 2$, the topologies $\wot$ and $\sot$ have the same comeager sets in the space of positive contractions on $\ell_p$. Since an \sot-typical positive contraction of $\ell_p$ is strongly stable by \cite[Prop. 3.1]{gillet1}, we deduce (by duality) that for every $1 < p < \infty$, a \wot-typical positive contraction on $\ell_p$ is weakly stable, implying that its weak limit semigroup is reduced to $\{0 \}$\footnote{The same remark holds for a $\wot$-typical contraction of $\ell_p$ with $1 < p \ne 2 < \infty$ by the works of Grivaux, Matheron and Menet \cite{GrivauxMatheronMenet21b} and \cite{GrivauxMatheronMenet22}.}. This explains why the study of the weak limit semigroup of $\wot$-typical positive contractions on $\ell^p$ is of limited interest.

\subsection{Positive operators on $L^p[0,1]$}\par\noindent

Section \ref{Section-genericposoperators} studies typical properties of positive contractions and isometries of $L^p[0,1]$ related to the weak limit semigroup, mainly for the $\wot$ and $\sot$ topologies. Thanks to Theorem \ref{thmtypicalWOTbijisometries}, we can show that the orbit of a positive contraction is meager in $(\Cp, \wot)$ for every $1 < p < \infty$. This was observed in \cite[Prop. 3.3]{EisnerMatrai13} in the class of Hilbert space contractions, whereas here it applies to every $L^p[0,1]$ with $1 < p < \infty$.

\begin{proposition} \label{thmmeagerorbitsposcontractions}
    Let $1 < p < \infty$. For every $T \in \Cp$, the orbit $\mathcal{O}_p(T) := \{ S T S^{-1} : S \in \U_p \}$ is meager in $(\Cp, \emph{\wot})$. In particular, $\emph{\wot}$-typical positive contractions of $L^p[0,1]$ are not conjugate in $\U_p$.
\end{proposition}

\begin{proof}
    It is enough to prove that the orbit is meager in $(\U_p, \sot)$. 
    
    Let us suppose first that $p = 2$. Let $\tau \in \GG$ and let $T = T_\tau^{(2)}$. In this case, the proof is analogous to \cite[Prop. 3.3]{EisnerMatrai13} and goes as follows. Let us denote by $\mu_{T}$ the maximal spectral type of $T$, which is a probability measure on $\T$. For every Borel probability measure $\nu$ on $\T$, the set $\{ S \in \U_2 : \mu_S \perp \nu \}$ is dense $G_\delta$ in $\U_2$, as observed in \cite[Thm. 8.30(b)]{Nadkarni-book}. Indeed, this set is easily seen to be $G_\delta$ in $\U_2$ since it is given by
    $$
    \displaystyle \bigcap_{m \geq 1} \{ S \in \U_2 : \exists \; U \; \textrm{open in $\T$}, \; \mu_S(U) < 1/m \; \textrm{and} \; \nu(U) > 1-1/m  \}
    $$ and the map $S \mapsto \mu_S$ is continuous from $\U_2$ into the space of probability measures on $\T$ (see \cite[Thm. 7.9 and Paragraph 8.22]{Nadkarni-book} for more details). Moreover, this last set is dense in $\U_2$. Indeed, this set is conjugacy invariant and we can always find an irrational rotation on $[0,1]$ whose maximal spectral type is singular with respect to $\nu$, showing the density in $\U_2$ by Proposition \ref{lemma-conjugacynonsingular}. It follows that the orbit $\mathcal{O}_2(T)$ is meager in $(\U_2, \sot)$. 

    The proof for any $1 < p < \infty$ follows from the case where $p = 2$. Indeed, for $1 \leq p_1 \leq p_2 < \infty$, the map $i_{p_1, p_2} : T_\sigma^{(p_1)} \mapsto T_\sigma^{(p_2)}$ from $(\U_{p_1}, \sot)$ onto $(\U_{p_2}, \sot)$ is a homeomorphic group isomorphism since the $p$-coarse topologies coincide on $\GG$. Since for every $\tau \in \GG$ the orbit $\mathcal{O}_2(T_\tau^{(2)})$ is meager in $(\U_2, \sot)$, it follows that the orbit $\mathcal{O}_{p}(T_\tau^{(p)}) = i_{2,p}(\mathcal{O}_{2}(T_\tau^{(2)}))$ is also meager in $(\U_p, \sot)$. 
    The second part of Proposition \ref{thmmeagerorbitsposcontractions} follows exactly as in \cite[Prop. 3.3]{EisnerMatrai13}. This
    concludes the proof of Proposition \ref{thmmeagerorbitsposcontractions}.
\end{proof}

Using that the map $i_{1,p} : T_\tau^{(1)} \mapsto T_\tau^{(p)}$ from $(\U_1, \sot)$ onto $(\U_p, \sot)$ is a homeomorphic group isomorphism for every $1 < p < \infty$, we obtain from Propositions \ref{proptypicalposisoisbijSOT} and \ref{thmmeagerorbitsposcontractions} the following analogue in $\I_1$.

\begin{proposition}
    For every $T \in \I_1$, the orbit $\mathcal{O}_1(T) := \{ S T S^{-1} : S \in \U_1 \}$ is meager in $(\I_1, \emph{\sot})$.
\end{proposition}

Beyond the study of the weak limit semigroup of typical positive contractions of $L^p[0,1]$ and motivated in part by the study of typical properties of contractions and positive contractions of $\ell_p$-spaces, as well as by the Invariant Subspace Problem for positive operators of $L^p[0,1]$, one may also investigate typical properties of positive contractions of $L^p[0,1]$ for other operator topologies.

\begin{problem}
    Study typical properties of positive contractions of $L^p[0,1]$ for the $\emph{\sot}$ and $\emph{\sote}$ topologies. 
\end{problem}

%1) 
%Other operators than polynomials? What about roots? Commuting operators (for $\mathcal{T}$ known for elements of $\mathcal{T}$, operators not)? 
%What is the weak closure of $\mathcal{T}$? Probably all not necessarily invertible transformations, but what else? For example the combinations of projection onto constants with $\lambda I$, $\lambda \in [0,1]$ by Stepin

%2) Is it true that for a rigid operator (in $\mathcal{I}$?), the commutant (in $\mathcal{I}$?) is either trivial (the powers + identity) or uncountable? (For transformations mentioned in King commutant rank one 1986, pages 364 below - 365 above, but in the mentioned reference [7] I could not find it)

%2) % Is it true that the set of embedding semigroups has typically cardinality continuum and every two elements of such are not isomorphic? (For transformations done in Stepin, Eremenko. Moreover, they prove that the centralizer of a generic transformation contains a subgroup isomorphic to an infinite-dimensional torus.)
%4) 
%Flows instead of transformations, $C_0$-semigroups instead of operators - what are weak limits?

%\section{Appendix: More on typicality}\label{sec:appendix}

\end{document}